\documentclass{amsart}

\usepackage[breaklinks]{hyperref}
\usepackage{amsmath}
\usepackage{amssymb}
\usepackage{mathrsfs}
\usepackage{amsthm}
\usepackage[utf8]{inputenc}
\usepackage{svninfo}
\usepackage{prettyref}
\ifx\FrenchText\undefined
\usepackage[british]{babel}
\else
\usepackage[french]{babel}
\AtBeginDocument{%
\catcode`\:=12%
\catcode`\!=12%
}
\fi

\makeatletter

\ifx\Presentation\undefined
\fi

\newcommand{\fref}[1]{\prettyref{#1}}
\newrefformat{item}{\ref{#1}}
\newrefformat{apx}{Appendix~\ref{#1}}
\newrefformat{cha}{Chapter~\ref{#1}}
\newrefformat{sec}{Section~\ref{#1}}

\newcounter{dummycnt}

\ifx\thmnum\undefined
\newtheorem{thmnum}{}[section]
\fi

\newcommand{\mynewthm}[3][]{%
  \def\PARAM{#1}
  \ifx\PARAM\empty
  \newtheorem{#2}[thmnum]{#3}
  \else
  \newtheorem{#2}{#3}[#1]
  \fi
  \newtheorem*{#2*}{#3}%
  \newrefformat{#2}{#3~\ref{##1}}%
}

\ifx\FrenchText\undefined
\newcommand{\ThmLabel}{Theorem}
\newcommand{\PrpLabel}{Proposition}
\newcommand{\LemLabel}{Lemma}
\newcommand{\FctLabel}{Fact}
\newcommand{\CorLabel}{Corollary}
\newcommand{\DfnLabel}{Definition}
\newcommand{\ConvLabel}{Convention}
\newcommand{\NtnLabel}{Notation}
\newcommand{\CstLabel}{Construction}
\newcommand{\ExmLabel}{Example}
\newcommand{\RmkLabel}{Remark}
\newcommand{\QstLabel}{Question}

\else
\newcommand{\ThmLabel}{\iflanguage{french}{Théorème}{Theorem}}
\newcommand{\PrpLabel}{Proposition}
\newcommand{\LemLabel}{\iflanguage{french}{Lemme}{Lemma}}
\newcommand{\FctLabel}{\iflanguage{french}{Fait}{Fact}}
\newcommand{\CorLabel}{\iflanguage{french}{Corollaire}{Corollary}}
\newcommand{\DfnLabel}{\iflanguage{french}{Définition}{Definition}}
\newcommand{\ConvLabel}{Convention}
\newcommand{\NtnLabel}{Notation}
\newcommand{\CstLabel}{Construction}
\newcommand{\ExmLabel}{\iflanguage{french}{Exemple}{Example}}
\newcommand{\RmkLabel}{\iflanguage{french}{Remarque}{Remark}}
\newcommand{\QstLabel}{Question}

\fi

\theoremstyle{plain}
\mynewthm{thm}{\ThmLabel}
\mynewthm{prp}{\PrpLabel}
\mynewthm{lem}{\LemLabel}
\mynewthm{fct}{\FctLabel}
\mynewthm{cor}{\CorLabel}

\theoremstyle{definition}
\mynewthm{dfn}{\DfnLabel}
\mynewthm{conv}{\ConvLabel}
\mynewthm{conj}{Conjecture}
\mynewthm{ntn}{\NtnLabel}
\mynewthm{cst}{\CstLabel}

\theoremstyle{remark}
\mynewthm{rmk}{\RmkLabel}
\mynewthm{qst}{\QstLabel}
\mynewthm{exm}{\ExmLabel}

\ifx\Presentation\undefined

\newcommand{\myenumlabel}[1]{\textnormal{(\roman{#1})}}
\renewcommand{\theenumi}{\myenumlabel{enumi}}
\fi

\newcounter{cycprfcnt}
\newenvironment{cycprf}%
{\begin{list}{\PackageWarning{begnac}{Label required for cycprf}}%
  {%
    \setcounter{cycprfcnt}{1}
    \setlength{\itemindent}{0.5\leftmargin}%
    \setlength{\leftmargin}{0pt}%
    \newcommand{\cpcurr}{\myenumlabel{cycprfcnt}}%
    \newcommand{\cpnext}{\addtocounter{cycprfcnt}{1}\cpcurr}%
    \newcommand{\cpprev}{\addtocounter{cycprfcnt}{-1}\cpcurr}%
    \newcommand{\cpnum}[1]{\setcounter{cycprfcnt}{##1}\cpcurr}%
    \newcommand{\cpfirst}{\cpnum{1}}%
    \newcommand{\impnext}{\cpcurr{} $\Longrightarrow$ \cpnext.}%
    \newcommand{\impfirst}{\cpcurr{} $\Longrightarrow$ \cpfirst.}%
    \newcommand{\eqnum}[2]{\cpnum{##1}{} $\Longleftrightarrow$ \cpnum{##2}.}%
  }%
}%
{\qedhere\end{list}}%

\def\indsym#1#2{%
  \setbox0=\hbox{$\m@th#1x$}%
  \kern\wd0%
  \hbox to 0pt{\hss$\m@th#1\mid$\hbox to 0pt{$\m@th#1^{#2}$\hss}\hss}%
  \lower.9\ht0\hbox to 0pt{\hss$\m@th#1\smile$\hss}%
  \kern\wd0}
\newcommand{\ind}[1][]{\mathop{\mathpalette\indsym{#1}}}

\def\nindsym#1#2{%
  \setbox0=\hbox{$\m@th#1x$}%
  \kern\wd0%
  \hbox to 0pt{\hss$\m@th#1\not$\kern1.4\wd0\hss}
  \hbox to 0pt{\hss$\m@th#1\mid$\hbox to 0pt{$\m@th#1^{#2}$\hss}\hss}%
  \lower.9\ht0\hbox to 0pt{\hss$\m@th#1\smile$\hss}%
  \kern\wd0}

\def\dotminussym#1#2{%
  \setbox0=\hbox{$\m@th#1-$}%
  \kern.5\wd0%
  \hbox to 0pt{\hss\hbox{$\m@th#1-$}\hss}%
  \raise.6\ht0\hbox to 0pt{\hss$\m@th#1.$\hss}%
  \kern.5\wd0}
\newcommand{\dotminus}{\mathbin{\mathpalette\dotminussym{}}}

\renewcommand{\emptyset}{\varnothing}

\def\models{\vDash}

\newcommand{\rest}{{\restriction}}

\newcommand{\sfrac}[2]{\hbox{$\frac{#1}{#2}$}}

\DeclareMathOperator{\tp}{tp}

\DeclareMathOperator{\Th}{Th}

\DeclareMathOperator{\tS}{S}
\DeclareMathOperator{\dcl}{dcl}
\DeclareMathOperator{\acl}{acl}

\DeclareMathOperator{\diam}{diam}

\DeclareMathOperator{\id}{id}

\DeclareMathOperator{\dom}{dom}

\DeclareMathOperator{\CB}{CB}
\DeclareMathOperator{\RM}{RM}
\DeclareMathOperator{\dM}{dM}

\newcommand{\fM}{\mathfrak{M}}

\newcommand{\cC}{\mathcal{C}}

\newcommand{\cE}{\mathcal{E}}

\newcommand{\cG}{\mathcal{G}}

\newcommand{\cL}{\mathcal{L}}

\newcommand{\cM}{\mathcal{M}}
\newcommand{\cN}{\mathcal{N}}

\newcommand{\cX}{\mathcal{X}}

\newcommand{\bN}{\mathbb{N}}

\newcommand{\bQ}{\mathbb{Q}}

\makeatother

\newcommand{\dyad}[2][n]{\hbox{$\frac{#2}{2^{#1}}$}}

\begin{document}

\title{Definability of groups in $\aleph_0$-stable metric structures}

\author{Itaï \textsc{Ben Yaacov}}

\address{Itaï \textsc{Ben Yaacov} \\
  Université Claude Bernard -- Lyon 1 \\
  Institut Camille Jordan, CNRS UMR 5208 \\
  43 boulevard du 11 novembre 1918 \\
  69622 Villeurbanne Cedex \\
  France}

\urladdr{\url{http://math.univ-lyon1.fr/~begnac/}}

\thanks{Author supported by Marie Curie research network ModNet,
  by ANR chaire d'excellence junior THEMODMET (ANR-06-CEXC-007)
  and by the Institut Universitaire de France.}

\svnInfo $Id: DefOSGrp.tex 974 2009-09-05 10:49:40Z begnac $
\thanks{\textit{Revision} {\svnInfoRevision} \textit{of} \today}

%\date{\today}

\begin{abstract}
  We prove that in a continuous $\aleph_0$-stable theory every
  type-definable group is definable.
  The two main ingredients in the proof are:
  \begin{enumerate}
  \item Results concerning Morley ranks (i.e., Cantor-Bendixson
    ranks) from \cite{BenYaacov:TopometricSpacesAndPerturbations},
    allowing us to prove the
    theorem in case the metric is invariant under the group action; and
  \item  Results concerning the existence of
    translation-invariant definable metrics on type-definable groups and
    the extension of partial definable metrics to total ones.
  \end{enumerate}
\end{abstract}

\maketitle

\section*{Introduction}

Definable sets, as well as more complex definable objects
(e.g., groups) play a central
and essential role in classical model theory.
These are (usually) subsets of the ambient structure which are defined
by a single formula of classical first order logic.

Continuous first order logic was proposed in
\cite{BenYaacov-Usvyatsov:CFO} as an
extension of classical first order logic, obtained by replacing the
two-element set of truth values $\{T,F\}$ with the compact interval
$[0,1]$.
It allows to consider various
classes of complete metric structures as elementary classes and to
study definability therein.
However, some things do become more complicated in continuous logic,
and in particular the classical
notion of a definable set splits in two.
First, a set is a predicate, and a definable set is a
definable predicate, i.e., a definable function into the
set $\{T,F\}$ (or $\{0,1\}$).
As such, the correct analogue is a
\emph{definable continuous predicate}, i.e., a definable function to
$[0,1]$ -- it is definable in the
sense that it is given by a continuous first order formula, or, at the
very least, by a uniform limit of such.
But when thinking of definable objects, such as groups, there is an
essential asymmetry between what is inside (which interests us) and what
is outside (about which we could hardly care less, especially if the set
is stably embedded).
The same asymmetry arises when we wish to quantify over a definable
set.
In that case the notion of a definable predicate is inadequate and we
are led to the notion of a \emph{definable set} in continuous logic:
this is a closed set
the distance to which is a definable predicate, or equivalently, over
which we may quantify (see \fref{fct:DefSet} below).

The class of definable set in a continuous structure is far less
well-behaved than in classical logic.
For example, the family of all definable subsets of $M^n$ does not
form a Boolean algebra, as it is not always closed under complement or
intersection.
Worse still, non trivial definable sets need not always even exist.
In particular, there exist theories which do not admit
\emph{enough definable sets}, i.e., where definable sets do not
suffice to separate types.

\begin{exm}
  Let us consider the theory
  $T = \Th(\bQ,0,1,+,\leq)$.
  Like any other classical theory, it may be viewed as
  a continuous theory by identifying the truth values $T$ and $F$
  with $0$ and $1$, respectively.
  Specifically, we add an axiom
  $\sup_{xy}  d(x,y)\wedge\neg d(x,y) = 0$
  asserting that the distance (i.e., equality) predicate takes values in
  $\{0,1\}$, and similarly for the predicate $\leq$.

  We now define a continuous predicate
  $P(x) = st\bigl( (x\wedge 1) \vee 0 \bigr)$, namely the standard
  part of $x$ truncated at $0$ from below and at $1$ from above.
  Notice that in any model of $T$ and for every $r \in [0,1]$,
  the conditions $P(x) \geq r$ and $P(x) \leq r$ are type-definable:
  \begin{gather*}
    P(x) \geq r
    \quad \Longleftrightarrow \quad
    \{ nx \geq m\colon n,m \in \bN, nr > m\}, \\
    P(x) \leq r
    \quad \Longleftrightarrow \quad
    \{ nx \leq m\colon n,m \in \bN, nr < m\}.
  \end{gather*}
  It follows that $P$ is a definable predicate in $T$.

  Let us now consider the continuous theory
  $T_P = \Th(\bQ,0,m_q,+,P)_{q\in\bQ}$
  where $m_q$ denotes multiplication by $q$.
  This is a reduct of $T$, and we leave it as
  an easy exercise to the reader to check that it admits quantifier
  elimination as well.
  For a member $a$ of a model of $T_P$ consider the following two
  values
  \begin{gather*}
    r^+ = \inf \{ q \in \bQ^{>0} \colon P(r/q) < 1 \}, \qquad
    r^- = \inf \{ q \in \bQ^{>0} \colon P(-r/q) < 1 \}.
  \end{gather*}
  Notice that either $P(a) = 0$ or $P(-a) = 0$,
  so at least one of $r^+$ and $r^-$ is zero and we may define
  $r = r^+ - r^- \in [-\infty,\infty]$.
  For every number $q \in \bQ$ we have:
  \begin{gather*}
    P(qa) =
    \begin{cases}
      1 & qr \geq 1, \\
      rq & 0 \leq qr \leq 1, \\
      0 & qr \leq 0.
    \end{cases}
    \qquad (\text{where } 0\cdot \infty = 0)
  \end{gather*}
  Thus $r$ determines the type of $a$.
  We obtain a bijection
  $\tS_1(T_P) \to [-\infty,\infty]$ which is easily checked to be
  continuous and therefore a homeomorphism.

  On the other hand, a model of $T_P$ carries the discrete
  $0/1$ metric.
  Thus, if $X$ is a definable set, then its complement is
  definable as well: $d(x,X^c) = \neg d(x,X)$.
  If $X$ is definable without parameters then the corresponding closed
  set $[X] \subseteq \tS_1(T_P)$ must be clopen, i.e.,
  either $\emptyset$ or all of $\tS_1(T_P)$.

  We conclude that a model of $T_P$ admits no non trivial
  $\emptyset$-definable sets (in $1$-space),
  so $T_P$ does not have enough definable sets.
\end{exm}

Since all known examples of this pathology are unstable it makes sense to
ask whether all stable continuous theories admit enough definable
sets.

One of the beautiful aspects of stable group theory in classical logic
is the proof that there are also ``enough definable groups'', namely,
that every type-definable group is the intersection of definable
subgroups of a definable group.
In the case of an $\aleph_0$-stable theory, chain conditions along with the
previous general fact yield that every type-definable group is
definable.
In continuous logic we can prove adequate analogues of the chain
conditions for sequences of definable (or type-definable) groups for
$\aleph_0$-stable theories, but we do not know how to prove enough
definable groups exist in stable theories.

In the present paper we give a direct proof of the fact that in an
$\aleph_0$-stable theory every type-definable group is definable, leaving
open the question of the existence of definable groups in general
stable theories.
In the special case of the theory of probability algebras this has
already been proved by Alexander Berenstein
\cite{Berenstein:DefinableSubgroupsOfMeasureAlgebras}.

In \fref{sec:Definability} we discuss various definability classes of
sets and functions.

In \fref{sec:MainThmApprox} we prove the main theorem
using some technical
results concerning Morley ranks (i.e., Cantor-Bendixson ranks)
from \cite{BenYaacov:TopometricSpacesAndPerturbations}.
We do this under the
assumption of the invariance of the metric under the group
operation (\fref{thm:MainThmApprox}).
The rest of the paper aims towards the removal of this assumption.

In \fref{sec:DefMet} we study definable metrics other than the
standard one.
In particular, we study when and how partial definable metrics (on a
definable or type-definable set) can be extended to total ones.

In \fref{sec:MainThmFull} we prove the full version of main theorem
and give some corollaries.

In \fref{sec:ChainConditions} we use several earlier results to
prove some chain conditions for \hbox{(type-)definable}
groups in stable and $\aleph_0$-stable theories.

\medskip

We assume familiarity with the development of continuous logic as
developed in the first half of \cite{BenYaacov-Usvyatsov:CFO}.
We shall also use facts regarding general stability and stable groups
from the second half of
\cite{BenYaacov-Usvyatsov:CFO} and from \cite{BenYaacov:StableGroups},
as well as regarding $\aleph_0$-stability and
topometric Cantor-Bendixson ranks from
\cite{BenYaacov:TopometricSpacesAndPerturbations}.

\section{Definability properties}
\label{sec:Definability}

This section consists mostly of definitions and relatively easy facts.
Some of the facts presented here also appear in
\cite[Section~9]{BenYaacov-Berenstein-Henson-Usvyatsov:NewtonMS}.

\subsection{Definability classes of sets}

\begin{dfn}
  \begin{enumerate}
  \item A \emph{type-definable set} $X$ is the set of
    realisations of an arbitrary set of conditions
    $\{\varphi_i(x) = 0\colon i < \lambda\}$.
  \item A \emph{zero set} $X$ is the set of realisations of a single
    condition $\psi(x) = 0$, where $\psi(x)$ is a definable
    predicate.
    We then say that $X$ is the zero set of $\psi$.
  \item A \emph{definable set} $X$ is a closed set for which
    $d(x,X)$ is a definable predicate.
    We say that $X$ is definable over a parameter set $A$,
    or that it is $A$-definable, if $d(x,X)$
    definable over $A$.
  \end{enumerate}
  Type-definable sets and zero sets will only be considered in
  sufficiently saturated structures, while definable sets make sense
  in an arbitrary structure.
\end{dfn}

Clearly, every zero set is type-definable, and every definable set $X$
is the zero set of $d(x,X)$.
In terms of types, we know that a type-definable set $X$ corresponds
to closed sets of types $[X] \subseteq \tS_n(A)$ where $A$ contains all the
parameters appearing in the definition of $X$.

\begin{dfn}
  Let $X$ and $Y$ be two type-definable sets.
  We say that $Y$ is a \emph{logical neighbourhood} of $X$, in symbols
  $X < Y$, if there is a set of parameters $A$ over which both $X$ and
  $Y$ are defined such that $[X] \subseteq [Y]^\circ$ in $\tS_n(A)$.
\end{dfn}

Notice that the interior of $[Y]$ does depend on $A$
(i.e., if $A' \supseteq A$ then $[Y]^\circ$
calculated in $\tS_n(A')$ may be larger than the
pullback of the interior of $[Y]$ in $\tS_n(A)$).
We may nonetheless choose any parameter set we wish:

\begin{lem}
  \label{lem:LogNeighb}
  Assume that $X$ is type-definable with parameters in $B$, $Y$
  type-definable possibly with additional parameters not in $B$.
  Then:
  \begin{enumerate}
  \item If $X < Y$ then $[X] \subseteq [Y]^\circ$ in $\tS_n(A)$
    for any set $A$ over which both
    $X$ and $Y$ are defined.
  \item If $X < Y$ then there is an intermediate logical neighbourhood
    $X < Z < Y$, which can moreover be taken to be the zero set of a formula
    with parameters in $B$.
  \item If $Y\cap X = \emptyset$ then there is a logical neighbourhood
    $Z > X$ such that $Z \cap Y = \emptyset$.
    Moreover, we may take $Z$ to be a zero set defined over $B$.
  \end{enumerate}
\end{lem}
\begin{proof}
  Assume $X < Y$, where $X$ is type-definable over $B$, and $Y$ over
  $A \supseteq B$.
  Let $\Phi$ consist of all formulae $\varphi(\bar x)$ over $B$ which
  are zero on $X$.
  If $\varphi,\psi \in \Phi$ then $\varphi\vee\psi \in \Phi$,
  and $X$ is defined by the partial type
  $p(\bar x)
  = \{\varphi(\bar x) \leq r\colon \varphi \in \Phi, r > 0\}$.
  By compactness in $\tS_n(A)$
  there is a condition $\varphi(\bar x) \leq r$ in $p(\bar x)$
  which already implies $\bar x \in Y$.
  Let $Z$ be the zero set of
  the formula $\varphi(\bar x) \dotminus r'$
  where $0 < r' = \frac{k}{2^{-m}} < r$.

  Then in $\tS_n(A)$ we have
  $[X] \subseteq [\varphi(\bar x) < r']
  \subseteq [\varphi(\bar x) \leq r']
  \subseteq [\varphi(\bar r) < r]
  \subseteq [Y]$,
  i.e.,
  $[X] \subseteq [Z]^\circ \subseteq [Z] \subseteq [Y]^\circ$,
  proving the first two items.
  The third item now follows from the fact that $\tS_n(A)$ is a normal
  topological space.
\end{proof}

\begin{lem}
  \label{lem:ZeroSetGd}
  A type-definable set $X$ is a zero set if and only if
  $[X]$ is a (closed) $G_\delta$ set.
\end{lem}
\begin{proof}
  This is just a topological statement, saying that in a compact
  Hausdorff space $Y$, a closed subset $K \subseteq Y$ is the zero set of some
  $f \in C(Y,[0,1])$ if and only if it is a $G_\delta$ set.
  This is in fact true in an arbitrary normal space: left to right is
  immediate, while right to left involves a straightforward
  construction using $\omega$ applications Urysohn's Lemma.
\end{proof}

It follows that finite unions and countable intersections of
zero sets are zero sets.
In particular, a set $X$ which is type-definable by a countable set
of conditions is a zero set.

Later on we shall use the following result:
\begin{lem}
  \label{lem:ZeroSetByCoords}
  Let $X = \prod_{i<n} X_i^{m_i}$ be a type-definable set, (so each $X_i$
  is one) and $Y \supseteq X$ a zero set.
  Then there are zero sets $Y_i \supseteq X_i$ such that
  $Y \supseteq \prod_{i<n} Y_i^{m_i}$.
\end{lem}
\begin{proof}
  We only show that if $X \times X'$ is a type-definable set and
  $Y \supseteq X \times X'$ is a zero set then there is a zero set
  $Y' \supseteq X'$ such
  that $Y \supseteq X \times Y'$.
  The result then follows since the intersection of finitely (or even
  countably) many zero sets is a zero set.

  Let $\varphi(x,y) = 0$ define $Y$.
  For $n < \omega$ consider the partial type
  $\{x \in X\} \cup \{y \in X'\} \cup \{\varphi(x,y) \geq 2^{-n}\}$.
  As it is inconsistent $Y$ admits a logical neighbourhood
  $Y_n > X'$ such that
  $\{x \in X\} \cup \{y \in Y_n\} \cup \{\varphi(x,y) \geq 2^{-n}\}$
  is inconsistent.
  Moreover, choosing the sets $Y_n$ by induction on $n$ we may arrange
  that $Y_n > Y_{n+1}$.
  Let $Y' = \bigcap Y_n$.
  Then $[Y'] = \bigcap [Y_n]$ is a closed $G_\delta$ set, and
  $Y \supseteq X\times Y'$.
\end{proof}

In many situations we may wish do show that if a condition holds on a
type-definable set then it holds on some zero set containing it.
This is (tautologically) the case if the condition itself is a zero set.
It is still true if the condition in question is a containment (i.e.,
implication) of zero sets.
\begin{lem}
  \label{lem:ZeroSetImplication}
  Let $\varphi(\bar x)$ and $\psi(\bar x)$ be two
  definable predicates, $X$ a type-definable set,
  and assume that for all
  $\bar x \in X$: $\varphi(\bar x) = 0 \Longrightarrow \psi(\bar x) = 0$.
  Then there exist a zero set
  $Y \supseteq X$ on which
  $\varphi(\bar x) = 0 \Longrightarrow \psi(\bar x) = 0$
  holds as well.
\end{lem}
\begin{proof}
  For every $\varepsilon > 0$ there is $\delta(\varepsilon) > 0$
  such that for all $\bar x \in X$:
  $\varphi(\bar x) < \delta(\varepsilon)
  \Longrightarrow \psi(\bar x) \leq \varepsilon$.
  Indeed, if not, then we can obtain a contradiction to our assumption
  using compactness.
  We can therefore take $Y$ to be the zero set of:
  \begin{gather*}
    \chi(\bar x) = \sum_{n<\omega} 2^{-n-1} \left(
      (\delta(2^{-n}) \dotminus \varphi(\bar x)) \wedge
      (\psi(\bar x) \dotminus 2^{-n})
    \right).
    \qedhere
  \end{gather*}
\end{proof}

Finally, when it comes to definable sets, there are several
important equivalent characterisations.
For a structure $\cM$, a subset $X \subseteq M^n$,
a definable predicate $\varphi(\bar x,\bar y)$ (possibly with
parameters in $M$) and $\bar b \in M^m$ define
\begin{gather*}
  \sup_{\bar x \in X} \varphi(\bar x,\bar b)
  =
  \sup\, \{ \varphi^\cM(\bar a,\bar b) \colon \bar a \in X\}.
  \qquad (\text{where } \sup \emptyset = 0)
\end{gather*}
Thus $\sup_{\bar x \in X} \varphi(\bar x,\bar y)$
is a $[0,1]$-valued predicate on $M^m$ which need not be definable.

\begin{fct}
  \label{fct:DefSet}
  Let $\cM$ be a structure, $X \subseteq M^n$ closed subset.
  Let also $A \subseteq M$ be a set of parameters.
  Then the following are equivalent:
  \begin{enumerate}
  \item $X$ is definable in $\cM$ over $A$.
  \item $X$ is the zero set in $\cM$ of an $A$-definable predicate
    $\psi(\bar x)$, and $d(\bar x,X) \leq \varphi(\bar x)$
    on $M^n$.
  \item \label{item:DefSetMetricNeighbourhood}
    For every $\varepsilon > 0$ there exists
    a formula $\psi_\varepsilon(\bar x)$ (with parameters in $A$)
    such that:
    \begin{gather*}
      X
      \subseteq
      \bigl\{ \bar a \in M^n\colon \psi_\varepsilon(\bar a) = 0 \bigr\}
      \subseteq
      \bigl\{ \bar a \in M^n\colon \psi_\varepsilon(\bar a) < 1 \bigr\}
      \subseteq
      B(X,\varepsilon).
    \end{gather*}
  \item For every $A$-definable predicate $\varphi(\bar x,\bar y)$,
    the predicate $\sup_{\bar x\in X} \varphi(\bar x,\bar y)$ is
    definable in $\cM$ over $A$.
  \item For every formula $\varphi(\bar x,\bar y)$,
    the predicate $\sup_{\bar x\in X} \varphi(\bar x,\bar y)$ is
    definable in $\cM$ over $A$.
    \setcounter{dummycnt}{\value{enumi}}
  \end{enumerate}
  If $\cM$ is $(|A|+\aleph_0)^+$-saturated and
  $X$ is type-definable in $\cM$ over $A$
  then these conditions are further equivalent to:
  \begin{enumerate}
    \setcounter{enumi}{\value{dummycnt}}
  \item For every $\varepsilon > 0$, the set
    $\bar B(X,\varepsilon)$
    (which is necessarily type-definable over $A$) is a logical
    neighbourhood of $X$.
  \end{enumerate}
\end{fct}
\begin{proof}
  \begin{cycprf}
  \item[\impnext]
    Take $\psi(\bar x) = d(\bar x,X)$.
  \item[\impnext]
    Take $\psi_\varepsilon
    = \dot k \psi = (k\psi)\wedge 1$,
    where $k > 1/\varepsilon$.
  \item[\impnext]
    It will be enough to show that
    $\sup_{\bar x \in X} \varphi(\bar x,\bar y)$
    admits arbitrarily good uniform approximations by $A$-definable
    predicates.

    Given $\varepsilon > 0$ there exist
    $\delta > 0$ and $k \in \bN$ such that
    \begin{gather*}
      d(\bar x,\bar x') < \delta
      \quad \Longrightarrow \quad
      |\varphi(\bar x,\bar y)-\varphi(\bar x',\bar y)|
      \leq \varepsilon.
    \end{gather*}
    Let
    $\zeta(\bar x,\bar y) =
    \varphi(\bar x,\bar y) \dotminus \psi_\delta(\bar x)$,
    and let us fix
    $\bar b \in M^m$.
    For $\bar a \in X$ we have
    $\zeta(\bar a,\bar b) = \varphi(\bar a,\bar b)$.
    For arbitrary $\bar a$ in $M^n$,
    if $\zeta(\bar a,\bar b) > 0$
    then necessarily $\psi_\delta(\bar a) < 1$,
    so there is $\bar a' \in X$, $d(\bar a,\bar a') < \delta$,
    whereby
    \begin{gather*}
      \zeta(\bar a,\bar b) \leq \varphi(\bar x,\bar y)
      \leq \sup_{\bar x \in X} \varphi(\bar x,\bar b) + \varepsilon.
    \end{gather*}
    We obtain the desired approximation
    \begin{gather*}
      \sup_{\bar x \in X} \varphi(\bar x,\bar b)
      \leq
      \sup_{\bar x} \zeta(\bar x,\bar y)
      \leq
      \sup_{\bar x \in X} \varphi(\bar x,\bar b) + \varepsilon.
    \end{gather*}
  \item[\impnext]
    Immediate.
  \item[\impfirst]
    $d(\bar x,X) = \inf_{\bar y \in X} d(\bar x,\bar y)$.
  \item[\eqnum{3}{6}]
    An easy application of Urysohn's Lemma
    (and of density of formulae among definable predicates)
    in $\tS_n(A)$.
  \end{cycprf}  
\end{proof}

If a definable predicate $\psi(\bar x)$
defines the distance to a definable set $X$
then $\psi$ is $1$-Lipschitz and
$\psi(\bar x) \geq d(\bar x,X)$,
where $X$ is necessarily the zero set of $\psi$.
It is also not difficult to see that these are sufficient conditions.
Now let $c$ denote the parameter for
the definable predicate $\psi(\bar x)$, which we may re-write as
$\psi(\bar x,c)$.
The $1$-Lipschitz condition is easily expressed as a property of $c$
in continuous logic by D1 below.
The condition that
$\psi(\bar x) \geq d(\bar x,X)$
can be written as
$
\forall\bar x\, \exists \bar y \,
\bigl(
\psi(\bar y,c) = 0
\, \& \,
d(\bar x,\bar y) \leq \psi(\bar x,c)
\bigr)
$,
which in continuous logic can only be expressed \emph{approximately},
as in D2.
\begin{align*}
  \tag{D1}
  & \sup_{\bar x,\bar y} \, \bigl(
  \psi(\bar x,z) \dotminus \psi(\bar y,z)
  \dotminus d(\bar x,\bar y) \bigr)
  = 0, \\
  \tag{D2}
  & \sup_{\bar x} \, \inf_{\bar y} \,
  \Bigl(
  \psi(\bar y,z) \vee \bigl(
  d(\bar x,\bar y) \dotminus
  \psi(\bar x,z)
  \bigr) \Bigr) = 0.
\end{align*}
It turns out that even though D2 is merely an approximate version of
what we wished to express, it is still enough.
(Compare with similar conditions given in
\cite[Theorem~9.12]{BenYaacov-Berenstein-Henson-Usvyatsov:NewtonMS}.)

\begin{prp}
  \label{prp:DefSetParam}
  Let $\Sigma_\psi(z)$ consist of conditions {\rm D1} and {\rm D2}.
  Let $\cM$ be any structure, $c \in M$
  (possibly in an imaginary sort -- for example the sort of canonical
  parameters for instances of $\psi$)
  and let $X_c \subseteq M^n$ be the zero set of
  $\psi(\bar x,c)$.
  Then $\psi(\bar x,c) = d(\bar x,X_c)$
  in $\cM$ if and only if
  $\cM \models \Sigma_\psi(c)$.

  Moreover, the quantification in \fref{fct:DefSet} uniform,
  meaning that for any definable predicate
  $\varphi(\bar x,\bar y,z)$ there is a definable predicate
  $\chi_{\psi,\varphi}(\bar y,z)$ (both without parameters),
  such that for every structure $\cM$ and every
  $c \in M$:
  \begin{gather*}
    \Sigma_\psi(c)
    \quad \Longrightarrow \quad
    \sup_{\bar x \in X_c} \, \varphi(\bar x,\bar y,c)
    =\chi_{\psi,\varphi}(\bar y,c).
  \end{gather*}
\end{prp}
\begin{proof}
  Left to right is clear, so we prove right to left.
  By D1,  $\psi(\bar x,c)$ is
  $1$-Lipschitz, so
  $\psi(\bar x,c) \leq d(\bar x,X_c)$.
  For the other inequality, let $\bar a \in M^n$ and
  let $\varepsilon > 0$.
  By D2 there exists $\bar a_0$ such that
  \begin{gather*}
    \psi(\bar a_0,c) < \varepsilon,
    \qquad
    d(\bar a,\bar a_0) < \psi(\bar a,c) + \varepsilon.
  \end{gather*}
  Proceeding by induction we construct a sequence
  $\{\bar a_n\}$ in $\cM$ verifying
  \begin{gather*}
    \psi(\bar a_{n+1},c) < 2^{-n-1}\varepsilon,
    \qquad
    d(\bar a_n,\bar a_{n+1})
    <
    \psi(\bar a_n,c) + 2^{-n-1}\varepsilon
    <
    3\cdot 2^{-n-1}\varepsilon
  \end{gather*}
  This sequence is Cauchy and converges in $\cM$ to some
  $\bar b$.
  Then $\psi(\bar b,c) = 0$ by continuity of $\psi(\bar x,c)$
  and
  \begin{gather*}
    d(\bar x,X_c)
    \leq
    d(\bar a,\bar b)
    <
    \psi(\bar a,c) + \varepsilon
    + 3\varepsilon \sum 2^{-n-1}
    <
    \psi(\bar a,c) + 4 \varepsilon.
  \end{gather*}
  Thus $d(\bar a,X_c) \leq \psi(\bar a,c)$, as desired.

  The moreover part is by inspection of the proof of
  \fref{fct:DefSet}.
\end{proof}

It follows definable sets, as well as quantification over definable
sets, are respected by elementary extensions and restrictions.

\begin{cor}
  Assume $A \subseteq \cM \preceq \cN$.
  If $X_1 \subseteq M^n$ is $A$-definable then there is a unique
  $A$-definable subset $X_2 \subseteq N^n$ such that
  $X_1 = X_2 \cap M^n$.
  Conversely, if
  $X_2 \subseteq N^n$ is $A$-definable in $\cN$ then
  $X_1 = X_2 \cap M^n$ is $A$-definable in $\cM$.

  Moreover, assume this is the case, and let
  $\varphi(\bar x,\bar y)$ be an $A$-definable predicate,
  so $\sup_{\bar x \in X_1} \varphi(\bar x,\bar y)$
  and  $\sup_{\bar x \in X_2} \varphi(\bar x,\bar y)$
  are $A$-definable in $\cM$ and in $\cN$, respectively.
  Then the latter is the unique interpretation in $\cN$
  of the former.
  In particular, $d(\bar x,X_2)$ is the unique interpretation of
  $d(\bar x,X_1)$ in $\cN$.
\end{cor}
\begin{proof}
  Immediate from \fref{prp:DefSetParam} and the fact that two
  $A$-definable predicates which agree on $\cM$, also agree on $\cN$.
\end{proof}

We may therefore refer to a definable set $X$
or to the definable predicates
$d(\bar x,X)$ and $\sup_{\bar x \in X} \varphi(\bar x,X)$
without specifying the ambient structure explicitly
(it just has to contain all the required parameters).
In the situation described above we may then write
$X_1 = X(\cM)$, $d(\bar x,X_1) = d(\bar x,X)^\cM$,
and so on.
It is also worthwhile to notice that every definable set $X$ admits an
imaginary \emph{canonical parameter}, or \emph{code}, namely
an imaginary element which is fixed precisely by those automorphisms
of a large homogeneous ambient structure which fix $X$ set-wise.
Indeed, the canonical parameter of $d(\bar x,X)$ will do.

\begin{lem}
  The product and union of two definable sets are definable.
\end{lem}
\begin{proof}
  If $X$ and $Y$ are definable then
  $d((x,y),X\times Y) = d(x,X) \vee d(y,Y)$, where we equip the product sort
  with the maximum metric.
  If they are in the same sort then
  $d(x,X\cup Y) = d(x,X) \wedge d(x,Y)$.
\end{proof}

A finite intersection of definable sets needs not be definable in
general.
When it comes to infinite unions, we propose two results.

\begin{lem}
  \label{lem:DefSetCover}
  Let $X$ be a type-definable set, and assume
  $X = \bigcup_{i<\alpha} X_i$ where $\{X_i\colon i < \alpha\}$
  is a possibly infinite (yet bounded) family of definable sets.
  Then $X$ is definable.
\end{lem}
\begin{proof}
  By \fref{fct:DefSet} all we need to check is that for every $\varepsilon>0$
  the set $\bar B(X,\varepsilon)$ is a logical neighbourhood of $X$.
  Indeed:
  \begin{gather*}
    \bar B(X,\varepsilon) \supseteq B(X,\varepsilon)
    = \bigcup_{i<\alpha} B(X_i,\varepsilon)
    \supseteq \bigcup_{i<\alpha} \bar B(X_i,\varepsilon/2),
    \intertext{whereby:}
    [\bar B(X,\varepsilon)]^\circ
    \supseteq \bigcup_{i<\alpha} [\bar B(X_i,\varepsilon/2)]^\circ
    \supseteq \bigcup_{i<\alpha} [X_i] = [X].
    \qedhere
  \end{gather*}
\end{proof}

Similarly, a definable union of definable sets is definable:
\begin{lem}
  Let $X_{\bar a}$ be a family of uniformly definable sets with
  parameters in a definable set $Y$.
  That is to say that there is a definable predicate
  $\varphi(\bar x,\bar y)$ such that
  $d(\bar x,\bar X_{\bar a}) = \varphi(\bar x,\bar a)$ for every
  $\bar a \in Y$.
  Then $Z = \bigcup_{\bar a \in Y} X_{\bar a}$ is definable.
\end{lem}
\begin{proof}
  First, $Z$ is a closed set by a simple compactness argument.
  Then we have:
  $d(\bar x,Z) = \inf_{\bar y \in Y} d(\bar x,X_{\bar y})
  = \inf_{\bar y \in Y} \varphi(\bar x, \bar y)$.
\end{proof}

\subsection{Partial definable predicates and functions}

We shall consider objects such as predicates and functions which are
only defined on some type-definable set.
\begin{dfn}
  Let $X$ be a type-definable set.
  \begin{enumerate}
  \item A \emph{partial type-definable predicate} $\psi(\bar x)$ on $X$,
    is given by a continuous mapping $\psi\colon [X] \to [0,1]$,
    where $[X]$ is the closed set of complete types corresponding to
    $X$.
    We call $X$ the \emph{domain} of $\psi$, denoted $\dom(\psi)$, and for
    $\bar a \in X$ we denote by $\psi(\bar a)$ the truth value
    $\psi(\tp(\bar a))$.
  \item 
    It is \emph{definable} on $X$ if it is the restriction of a
    definable predicate to $X$.
  \item 
    A \emph{partial type-definable function} $f(\bar x)$ on $X$
    is one whose graph is given by a partial type
    $\cG_f(\bar x,y)$.
    That is to say that
    $\cG_f(\bar x,y) \cup \cG_f(\bar x,z) \models y = z$ and that
    $\dom(f) = X$ is defined by the partial type
    $\exists y\,\cG_f(\bar x,y)$.
    With a slight abuse of notation we may write it as
    $f\colon X \to M$ or $f\colon X(M) \to M$, although the model $M$
    (and even its complete theory) may vary.
  \item 
    It is \emph{definable} on $X$ if the predicate
    $d(f(\bar x),y)$ is definable on $X \times M$.
  \end{enumerate}
  We may allow parameters in the definitions by naming them in the
  language.
  Also, when explicitly specifying the domain we shall usually drop the
  qualifier ``partial''.
\end{dfn}

It is a straightforward exercise to verify that every partial
definable predicate or function is type-definable.
We wish to prove the converse.

We start with a fact from general topology concerning the extensions
of continuous functions.
\begin{fct}[Tietze's Extension Theorem]
  Let $X$ be a normal space, $C \subseteq X$ closed, and let
  $f\colon C \to [0,1]$ be a continuous function.
  Then $f$ admits an extension to a continuous function
  $\tilde f\colon X \to [0,1]$.
\end{fct}

\begin{lem}
  \label{lem:DefFuncSubst}
  Let $\psi(x_{<n},y)$ be a definable predicate,
  and $f\colon X \to M$ a type-definable partial $n$-ary function.
  Then $\psi(\bar x,f(\bar x))$ is a partial type-definable predicate on
  $X$.
\end{lem}
\begin{proof}
  Assume everything is defined without parameters,
  and let $K = [X] \subseteq \tS_n(T)$.
  Then $f$ induces a continuous function $\hat f\colon K \to \tS_{n+1}(T)$
  sending $\tp(\bar a) \mapsto \tp(\bar a,f(\bar a))$, and
  $\psi(\bar x,f(\bar x))$ is given by the composition
  $\psi \circ \hat f\colon K \to [0,1]$.
\end{proof}

\begin{prp}
  Every partial type-definable predicate or function is definable.
\end{prp}
\begin{proof}
  For predicates, this is just Tietze's Extension Theorem.
  For functions, assume that $f\colon X \to M$ is a type-definable function.
  Then $f'\colon X\times M \to M$ defined by $f'(\bar x,y) = f(\bar x)$ is
  type-definable as well.
  Let $\psi(\bar x,y,z) = d(y,z)$.
  Then $\psi(\bar x,y,f'(\bar x,y)) = d(y,f(\bar x))$ is a
  type-definable, and therefore definable,
  predicate on $X \times M$.
\end{proof}

This means that of the notions defined above we only need to retain
those of partial definable functions and predicates.
Moreover, while partial definable predicates are not entirely
superfluous, most of the time we shall avoid them, replacing any such
partial predicate with an (arbitrary) total definable predicate
extending it.
\begin{ntn}
  By the notation $\varphi(\bar x) \sqsupseteq \psi(\bar x)$ we mean
  that $\varphi(\bar x)$ is
  a definable predicate extending a partial definable predicate
  $\psi(\bar x)$.
\end{ntn}

\begin{rmk}
  One can add to the language a sort $S_{[0,1]}$ for the interval
  $[0,1]$, along with the tautological predicate
  $\id_{[0,1]}\colon S_{[0,1]} \to [0,1]$ (such a sort usually exists in
  $\cL^{eq}$).
  Then a partial predicate on a type-definable set $X$ is
  definable if and only if factors through a definable function to
  $S_{[0,1]}$.
\end{rmk}

If $Y \subseteq X$ are type-definable, we may say that $Y$ is definable
\emph{relative to $X$} if $d(x,Y)$ is a partial definable predicate on
$X$.
\begin{lem}
  If $X$ is definable and $Y \subseteq X$ is definable relative to $X$
  then $Y$ is definable.
\end{lem}
\begin{proof}
  Let $\varphi(x) \sqsupseteq d(x,Y)$ (where $d(x,Y)$ is defined on $X$).
  Then for all $x$:
  $d(x,Y) = \inf_{y \in X} \big( d(x,y) \dotplus \varphi(y) \big)$.
\end{proof}

Since we do not always know how to extend a partial definable function
to a total one, it is worthwhile to notice the following fact:
\begin{lem}
  \label{lem:DefSetImage}
  Let $X$ be a definable set and let $f$ be a partial definable
  function whose domain contains $X$.
  Then $f(X)$ is definable as well.
\end{lem}
\begin{proof}
  Let $\varphi(x,y) \sqsupseteq d(x,f(y))$.
  Then $d(x,f(X)) = \inf_{y\in X} \varphi(x,y)$, which is a definable predicate
  since $X$ is definable, whereby $f(X)$ is.
\end{proof}

It is usually fairly easy to reduce questions about arbitrary
type-definable sets to questions about zero sets.
For example, while it is not always possible to extend a partial
definable function
to a total one, one can always extend it to a zero set containing its
domain:
\begin{lem}
  \label{lem:PartFuncZeroSetExt}
  Let $X$ be a type-definable set, $f\colon X \to M$ a definable function on
  $X$.
  Then there is a zero set $Y \supseteq X$ such that $f$ extends to a
  definable function on $Y$.

  Moreover, for every choice of
  $\varphi(\bar x,y) \sqsupseteq d(f(\bar x),y)$
  there is a zero set $Y \supseteq X$ such that
  $\varphi\rest_{Y\times M}$ defines the graph of
  a partial definable function $f'\colon Y\to M$ (which extends $f$).
\end{lem}
\begin{proof}
  Let $\varphi(\bar x,y) \sqsupseteq d(f(\bar x),y)$, and let:
  \begin{gather*}
    \psi(\bar x,y,z) =
    \big(
    d(y,z) \dotminus \varphi(\bar x,y) \dotminus \varphi(\bar x,z)
    \big)
    \vee
    \big(
    \varphi(\bar x,y) \dotminus d(y,z) \dotminus \varphi(\bar x,z)
    \big)
    \vee
    \inf_t \varphi(\bar x,t)
  \end{gather*}
  Then $\psi$ is zero on $X\times M^2$, and by \fref{lem:ZeroSetByCoords}
  there is a zero set $Y \supseteq X$ such that $\psi$ is zero on
  $Y \times M^2$ as well.
  This means that for all $\bar x \in Y$ there is $y_0$ such that
  $\varphi(\bar x,y_0) = 0$, and that for any other $y$ one has
  $\varphi(\bar x,y) = d(y,y_0)$.
  Thus $\varphi(\bar x,y) = d(f'(\bar x),y)$ for some function
  $f'\colon Y \to M$ extending $f$.
\end{proof}

If the original type-definable set is closed under the
function(s) we wish to extend, we can make
sure that so is the extension:
\begin{prp}
  \label{prp:PartFuncClosedZeroSetExt}
  Assume that we are given:
  \begin{enumerate}
  \item For each $i < n$ sets $X_i \subseteq Y_i$ where $X_i$ is
    type-definable and $Y_i$ is a zero set.
  \item An ordinal $\alpha \leq \omega$, and for each
    $j < \alpha$ a partial definable
    function $f_j\colon \prod_{i<n} X_i^{m_{i,j}} \to X_{\ell_j}$.
  \end{enumerate}
  Then there are zero sets $Y_i \supseteq Z_i \supseteq X_i$
  and partial definable
  functions $g_j\colon \prod_{i<n} Z_i^{m_{i,j}} \to Z_{\ell_j}$
  extending $f_j$.

  Moreover, if we are given
  definable predicate
  $\varphi_j(\bar x_j,\bar y_j) \sqsupseteq d(f_j(\bar x_j),y_j)$ then we can
  arrange that $\varphi_j(\bar x_j,y_j) \sqsupseteq d(g_j(\bar x_j),y_j)$ as
  well.
\end{prp}
\begin{proof}
  For simplicity of notation we shall consider the special case of
  a single function $f\colon X^m \to X \subseteq M$, as the
  general case is identical (but with a lot more indexes).

  Let $\varphi(\bar x,y) \sqsupseteq d(f(\bar x),y)$ be given (or just
  choose one).
  By \fref{lem:PartFuncZeroSetExt}, and possibly replacing
  the given zero set $Y \supseteq X$ with a
  smaller zero set, we may assume there is a partial definable
  function $g\colon Y^m \to M$ such that
  $\varphi(\bar x,y) \sqsupseteq g(f(\bar x),y)$.

  Let $Y_0 = Y$.
  Given a zero set $Y_k$ satisfying $Y \supseteq Y_k \supseteq X$ let
  $W_k = g^{-1}(Y_k)\cap Y_k^m$.
  Then $W_k \supseteq X^m$ is a zero set and by
  \fref{lem:ZeroSetByCoords} we can find a zero set
  $Y_{k+1} \supseteq X$ such that $W \supseteq Y_{k+1}^m$, i.e., such that
  $Y_{k+1} \subseteq Y_k$ and $g(Y_{k+1}^m) \subseteq Y_k$.

  In this manner we construct a countable decreasing sequence of zero
  sets
  $Y = Y_0 \supseteq Y_1 \supseteq \ldots \supseteq Y_k
  \supseteq \ldots \supseteq X$
  such that $g(Y_{k+1}^m) \subseteq Y_k$.
  Then $Z = \bigcap_k Y_k$ is a zero set,
  $Y \supseteq Z \supseteq X$ and $g(Z^m) \subseteq Z$,
  as desired.
\end{proof}

On the other hand,
in later section we shall have to consider logical neighbourhoods of
domains of partial definable functions (specifically: logical
neighbourhoods of type-definable groups, on which the group law is a
partial definable function).
While the function does not necessarily extend as such, we can extend
it as a multi-valued function, which is in addition approximately
well-defined on small enough neighbourhoods of the original domain.
\begin{lem}
  \label{lem:ApproxPartFuncNbhd}
  Let $X$ be a type-definable set, $f\colon X \to M$ a partial definable
  function.
  Then there is a definable predicate
  $\varphi_f(\bar x,y) \sqsupseteq d(f(\bar x),y)$ satisfying in addition
  $\sup_{\bar x} \inf_y \varphi_f(\bar x,y) = 0$.
  Letting $\tilde f(\bar x) = \{y\colon \varphi_f(\bar x,y) = 0\}$
  we have:
  \begin{enumerate}
  \item For all $\bar x$:
    $\tilde f(\bar x) \neq \emptyset$;
  \item If $\bar x \in X$ then
    $\tilde f(\bar x) = \{f(\bar x)\}$;
  \item For every $\varepsilon > 0$ there is a logical neighbourhood $Y > X$
    such that for all $\bar x \in Y$:
    $\diam(\tilde f(\bar x)) \leq \varepsilon$.
  \end{enumerate}
\end{lem}
\begin{proof}
  First choose any $\varphi_{f,0}(\bar x,y) \sqsupseteq d(f(\bar x),y)$.
  Then define:
  \begin{gather*}
    \varphi_f(\bar x,y)
    = \varphi_{f,0}(\bar x,y) - \inf_z \varphi_{f,0}(\bar x,z).
  \end{gather*}
  As $\bar x \in X$ implies
  $\inf_z \varphi_{f,0}(\bar x,z) = 0$ we still have
  $\varphi_f(\bar x,y) \sqsupseteq d(f(\bar x),y)$ whence the second item.
  On the other hand
  $\sup_{\bar x} \inf_y \varphi_f(\bar x,y) = 0$ follows from the
  definition and implies the first item.

  Finally, let $\Sigma(\bar x,y,z)$ denote the partial type saying
  that
  $\{y,z \in \tilde f(\bar x)\} \cup \{d(y,z) \geq \varepsilon\}$.
  Then $\Sigma \cup \{\bar x \in X\}$
  is inconsistent by the second item, and by compactness
  $\Sigma \cup \{\bar x \in Y\}$
  is inconsistent for some
  logical neighbourhood $Y > X$.
\end{proof}

In classical logic every partial definable function on a
type-definable set $X$ can be extended to a definable set containing
$X$.
Whether this is true in continuous logic is still open, but we can
show this up to an extension of the sort on which the function acts
via an isometric embedding in a larger sort.

\begin{lem}
  \label{lem:PartFuncRangeExt}
  Let $S_1$ and $S_2$ be two (imaginary) sorts, $X \subseteq S_1$
  type-definable, and $f\colon X \to S_2$ a partial definable function
  (usually we would have $S_1 = M^n$ and $S_2 = M^m$ being powers
  of the home sort).
  Then:
  \begin{enumerate}
  \item We can find $\varphi(x,y) \sqsupseteq d(f(x),y)$ satisfying:
    \begin{gather}
      \label{eq:DistFunc}
      \varphi(x,y) - \varphi(x,y')
      \leq d(y,y')
      \leq \varphi(x,y) + \varphi(x,y').
    \end{gather}
  \item There exists an imaginary sort $S_3$, a definable isometric
    embedding
    $\theta\colon S_2 \hookrightarrow S_3$,
    and a total definable function
    $\hat f\colon S_1 \to S_3$ extending $\theta\circ f$,
    such that moreover
    $\varphi(x,y) = d(\hat f(x),\theta(y))$.
  \end{enumerate}
\end{lem}
\begin{proof}
  Choose $\varphi_0(x,y) \sqsupseteq d(f(x),y)$
  and let $\varphi(x,y) = \sup_z |\varphi_0(x,z)-d(z,y)|$.
  Notice that for $x \in X$ we have
  $\varphi(x,y) = \sup_z |d(f(x),z)-d(z,y)| = d(f(x),y)$, i.e.,
  $\varphi(x,y) \sqsupseteq d(f(x),y)$ as well.

  Let $\chi(t)$ be a formula for which $0$ and $1$ are possible truth
  values, and let
  \begin{gather*}
    \psi(z,xyt) = \chi(t)\wedge \varphi_0(x,z) + (\neg\chi(t))\wedge d(y,z).
  \end{gather*}
  Let $S_3 = \{[xyt]\colon x,t \in S_1, y \in S_2\}$ be the sort of
  canonical parameters of instances of $\psi(z,xyt)$.
  We recall that the metric on it is given by:
  $d([xyt],[x'y't'])
  = \sup_z |\psi(z,xyt)-\psi(z,x'y't')|$.

  The embedding $\theta\colon S_2 \hookrightarrow S_3$ is given by
  $y \mapsto [x_0y\tilde 0]$ where
  $\chi(\tilde 0) = 0$ and $x_0 \in S_1$ is arbitrary.
  This is indeed isometric and does not depend on the choice of either
  $\tilde 0$ or $x_0$ as
  $\psi(z,x_0y\tilde 0) = d(y,z)$.
  Thus $d(\theta(y),\theta(y')) = \sup_z |d(y,z)-d(y',z)| = d(y,z)$.

  Similarly we define $\hat f\colon S_1 \to S_3$ by
  $x \mapsto [xy_0\tilde 1]$ where $\chi(\tilde 1) = 1$,
  obtaining $\psi(z,xy_0\tilde 1) = \varphi_0(x,z)$.
  In particular for $x \in S_1$ and $y \in S_2$ we have
  $d(\hat f(x),\theta(y)) = \varphi(x,y)$ as desired, from which
  follows \fref{eq:DistFunc}.

  Finally, if $x \in X$ then
  $d(\theta\circ f(x),\theta(y))
  = d(f(x),y)
  = \varphi(x,y)
  = d(\hat f(x),\theta(y))$, so
  $\hat f$ extends $\theta\circ f$ and the proof is complete.
\end{proof}

\begin{qst}
  Given a type-definable set $X$ and a partial definable function
  $f\colon X \to M$, can one find a definable set $Y \supseteq X$ such that $f$
  extends to a definable $\hat f\colon Y \to M$?
  This is true in classical logic.
\end{qst}

\section{The main theorem (first approximation)}
\label{sec:MainThmApprox}

Once we have defined (type-)definable sets and functions we
automatically have corresponding notions for more complex algebraic
structures.
For example:

\begin{dfn}
  By a \emph{type-definable group} we mean a type-definable set $G$
  equipped with a definable function $\cdot \colon G^2 \to G$ defining a group
  law on $G$.

  It is \emph{definable} if the set $G$ is definable.
\end{dfn}

Notice that if $\langle G,\cdot\rangle$
is a type-definable group with parameters in
$A$ then its identity $e_G$ belongs to $\dcl(A)$: indeed, every
automorphism of the universal domain fixing $A$ would fix $e_G$.
Similarly, the function $x \mapsto x^{-1}$ is definable on $G$ with
parameters in $A$, its graph
given by $\cG_{x^{-1}}(x,y) = \cG_\cdot(x,y,e_G)$.

The main goal of this article is to give some sufficient conditions
under which type-definable groups are definable.
For example, we have already essentially proved:
\begin{prp}
  \label{prp:BddIdxDefGrp}
  Let $G$ be a type-definable group, $H \leq G$ a subgroup of bounded
  index, and assume $H$ is definable.
  Then so is $G$.
\end{prp}
\begin{proof}
  Say $G/H = \{g_iH\colon i < \lambda\}$.
  Each coset $g_iH$ is definable as the image of $H$ under the partial
  definable function $x \mapsto g_ix$ (\fref{lem:DefSetImage}).
  Then by \fref{lem:DefSetCover} $G = \bigcup_{i<\lambda} g_iH$ implies that
  $G$ is definable.
\end{proof}

\subsection{Invariant metrics}

The main theorem states that in an $\aleph_0$-stable theory, every
type-definable group is definable.
As a first approximation, we prove this under the assumption that
the metric is invariant under the group operation:

\begin{dfn}
  A metric defined on a type-definable group is \emph{left-invariant}
  (\emph{right-invariant}) if it
  is invariant under left (right) translation.
  It is \emph{invariant} if it is both left- and right-invariant.
  It is \emph{inverse-invariant} if it is invariant under
  $x \mapsto x^{-1}$.
\end{dfn}
Clearly if the metric is left-invariant (or right-invariant) and
inverse-invariant then it is right-invariant (or left-invariant) and
therefore invariant.
Conversely, if the metric is invariant then it is in particular
inverse-invariant, as we have:
$d(x^{-1},y^{-1}) = d(x^{-1}y,e) = d(y,x) = d(x,y)$.

In classical first order logic it is an easy consequence of
compactness that on every type-definable group, the group law and
inverse can be extended to be well-defined (and make some sense) on
some definable set containing the group, i.e., on some
logical neighbourhood of the group.
In the continuous sense things are trickier, and the best we can hope
for is a logical neighbourhood on which an \emph{approximate} product
is \emph{approximately} well-defined and (in case the metric
on the group is invariant) approximately isometric.

\begin{lem}
  \label{lem:ApproxProdNbhd}
  Let $G$ be a type-definable group on which the metric is invariant.
  Let $\varphi_\cdot(x,y,z) \sqsupseteq d(xy,z)$
  be as in \fref{lem:ApproxPartFuncNbhd},
  $x \mathop{\tilde\cdot} y = \{z\colon \varphi_.(x,y,z) = 0\}$.

  Then for every $\varepsilon > 0$ and logical neighbourhood $X > G$ there is an
  intermediate logical
  neighbourhood $X > Y > G$ such that
  $Y \mathop{\tilde\cdot} Y \subseteq X$,
  and multiplication is almost isometric in the
  sense that for all $x,y,y' \in Y$ and for all
  $z \in x \mathop{\tilde\cdot} y$,
  $z' \in x \mathop{\tilde\cdot} y'$
  (or $z \in y \mathop{\tilde\cdot} x$,
  $z' \in y' \mathop{\tilde\cdot} x$):
  $|d(y,y') - d(z,z')| \leq \varepsilon$.
\end{lem}
\begin{proof}
  Since $X > G$ we can find a formula $\chi_{G,X}(x)$ which is equal to $0$ on
  $G$ and to $1$ outside $X$.
  Then all the following partial types are contradictory:
  \begin{align*}
    & x,y \in G,   \, z\in x \mathop{\tilde\cdot} y, \,
    \chi_{G,X}(z) = 1, \\
    & x,y,y' \in G,\, z\in x \mathop{\tilde\cdot} y, \,
    z' \in x\mathop{\tilde\cdot} y', \,
    |d(y,y') - d(z,z')| \geq \varepsilon, \\
    & x,y,y' \in G,\, z\in y \mathop{\tilde\cdot} x, \,
    z' \in y'\mathop{\tilde\cdot} x, \,
    |d(y,y') - d(z,z')| \geq \varepsilon.
  \end{align*}
  By compactness there exist a logical neighbourhood $X > Y > G$
  such that they are still contradictory when
  $G$ is replaced everywhere with $Y$, and this $Y$ will do.
\end{proof}

\subsection{Cantor-Bendixson and Morley ranks}

We shall use $\aleph_0$-stability via Morley ranks, i.e., Cantor-Bendixson
ranks in $\tS_n(M)$ where $\cM$ is a sufficiently saturated model.
Such ranks were studied in detail in
\cite{BenYaacov:TopometricSpacesAndPerturbations} in the
general setting of topometric spaces.
In fact, that paper discusses several possible notions of
Cantor-Bendixson ranks, of which we use one.

\begin{ntn}
  In this paper, $\CB_\varepsilon$ will denote what is denoted in
  \cite{BenYaacov:TopometricSpacesAndPerturbations}
  by $\CB_{f,\varepsilon}$, i.e., the Cantor-Bendixson
  rank based on the removal of open $\varepsilon$-finite sets.
\end{ntn}

\begin{dfn}
  Let $X$ be a type-definable set of $n$-tuples.
  We define the \emph{$\varepsilon$-Morley rank} of $X$
  as $\RM_\varepsilon(X) = \CB^{\tS_n(M)}_\varepsilon(X)$,
  where $\cM$ is any sufficiently
  saturated model containing the parameters needed for $X$
  (this does not depend on the choice of $\cM$).
  If $\alpha = \RM_\varepsilon(X)$ then $[X] \cap
  \tS_n(M)^{(\alpha)}_\varepsilon$
  is $\varepsilon$-finite, i.e., $\varepsilon$-$k$-finite for some
  $k$, and we define the \emph{$\varepsilon$-Morley degree} of $X$ to
  be $\dM_\varepsilon(X) = k$.
\end{dfn}

\begin{lem}
  \label{lem:ConstRM}
  Let $X$ be a type-definable set.
  For all $r > 0$ there are $0 < \varepsilon < r' < r$ such that
  $\RM_{r'}(X) = \RM_{r'-\varepsilon}(X)$.
\end{lem}
\begin{proof}
  For $n < \omega$ define $r_n = r(1-2^{-n-1})$.
  Then $(r_n\colon n < \omega)$ is an increasing sequence,
  whereby $\RM_{r_n}(X)$ is a decreasing sequence of ordinals and
  therefore stabilises from some point onwards.
  Thus we may take $r' = r_{n+1}$ and $\varepsilon =  r_{n+1}-r_n$
  for any $n$ large enough.
\end{proof}

We recall some definitions and the main result we use
from \cite[Section~3.3]{BenYaacov:TopometricSpacesAndPerturbations}.
\begin{ntn}
  Let $X,Y$ be two compact spaces, $R \subseteq X \times Y$ a closed relation.
  For $x \in X$ and $A \subseteq Y$ we define:
  \begin{align*}
    R_x & = \{y \in Y\colon (x,y) \in R\}, \\
    R^{\forall A} & = \{x \in X\colon R_x \subseteq A\}, \\
    R^{\exists A} & = \{x \in X\colon R_x \cap A \neq \emptyset\}.
  \end{align*}
\end{ntn}

\begin{fct}
  \label{fct:RelCBRank}
  Let $X,Y$ be two compact topometric spaces,
  $R \subseteq X \times Y$ a closed relation,
  and $\varepsilon,\delta > 0$ such that for all
  $(x,y),(x',y') \in R$:
  if $d_Y(y,y') \leq \delta$ then $d_X(x,x') \leq \varepsilon$.
  Let $K \subseteq X$ and $F \subseteq Y$ be closed sets such that
  $K \subseteq (R^{\exists Y})^\circ \cap R^{\forall F}$.
  Then $\CB^X_\varepsilon(K) \leq \CB^Y_\delta(F)$.
\end{fct}
\begin{proof}
  \cite[Theorem~3.23]{BenYaacov:TopometricSpacesAndPerturbations}.
\end{proof}

The following result contains the technical core of the proof of the
(first approximation of the) main theorem.
Given a logical neighbourhood of $G$ on which product is well-behaved
and a point in that neighbourhood (outside $G$), we can approximately
translate $G$ by that element, obtaining an approximately isometric
copy of $G$.
Using the Fact cited above we can compare the Morley ranks of $G$ and
of this approximate copy.
In addition, if the element we translate by is far enough from $G$,
then so will be the entire copy.

\begin{lem}
  \label{lem:ApproxCopyRM}
  Let $G$ be a group on which the metric is invariant.

  Then for every $\varepsilon > 0$ and logical neighbourhood $X > G$ there
  exists an intermediary logical neighbourhood $X > Y > G$
  such that for all $r > \varepsilon$
  either $Y \subseteq \bar B(G,r)$ or
  there is a type-definable subset $Z \subseteq X$ such that
  $d(Z,G) > r - \varepsilon$ and
  $\RM_{r-\varepsilon}(Z) \geq \RM_r(G)$.
\end{lem}
\begin{proof}
  Apply \fref{lem:ApproxProdNbhd} to $X > G$ and $\varepsilon > 0$ to obtain
  $X > Y > G$.
  If $Y \subseteq \bar B(G,r)$ then we are done, so assume the contrary,
  i.e., that there exists $y_0 \in Y$ such
  that $d(y_0,G) > r$.
  Set $Z = y_0\mathop{\tilde\cdot} G \subseteq X$, which is type-definable.

  To see that $d(G,Z) > r-\varepsilon$, let $g \in G$ and $z \in Z$.
  Then $z \in y_0\mathop{\tilde\cdot} h$ for some $h \in G$.
  Let $g' = gh^{-1} \in G$, so $d(y_0,g') > r$.
  Since $y_0,g',h \in Y$ we have
  $|d(z,g) - d(y_0,g')| \leq \varepsilon$, whereby
  $d(z,g) > r-\varepsilon$.
  By a compactness argument it follows that
  $d(G,Z) > r-\varepsilon$.
  
  Let $\cM$ be a fairly saturated model containing all necessary
  parameters, including $y_0$, and let $S = \tS(M)$.
  Set
  \begin{gather*}
    R = \{(\tp(u/M),\tp(v/M))\colon u \in Y, v \in y_0\mathop{\tilde\cdot} u\}.
  \end{gather*}
  Then $R \subseteq S^2$ is a closed relation,
  $[Y] = R^{\exists S}$ and $[G] \subseteq R^{\forall[Z]}$, so
  $[G] \subseteq (R^{\exists S})^\circ \cap R^{\forall[Z]}$.
  On the other, if $(p,q),(p'q') \in R$ then
  $|d(p,p')-d(q,q')| \leq \varepsilon$,
  so $d(q,q') \leq r-\varepsilon \Longrightarrow d(p,p') \leq r$.
  Applying \fref{fct:RelCBRank} we get:
  \begin{gather*}
    \RM_r(G)
    = \CB^S_r([G]) \leq \CB^S_{r-\varepsilon}([Z])
    = \RM_{r-\varepsilon}(Z).
    \qedhere
  \end{gather*}
\end{proof}

\subsection{The definability proof}

Notice that we have not yet used the assumption of
$\aleph_0$-stability (the ranks in \fref{lem:ApproxCopyRM} may well be
infinite).

\begin{thm}
  \label{thm:MainThmApprox}
  Let $G$ be a type-definable group in an $\aleph_0$-stable theory
  on which the metric is invariant.
  Then $G$ is definable.
\end{thm}
\begin{proof}
  Let $r > 0$, and we shall show that $\bar B(G,r) > G$.
  There is no harm if we decrease $r$, so
  by \fref{lem:ConstRM} we may assume that
  $\RM_r(G) = \RM_{r-\varepsilon}(G)$ for some $0 < \varepsilon < r$.

  By compactness we can find $X > G$ such that
  $\RM_{r-\varepsilon}(G) = \RM_{r-\varepsilon}(X)$ and
  $\dM_{r-\varepsilon}(G) = \dM_{r-\varepsilon}(X)$.
  Apply \fref{lem:ApproxCopyRM} to find $X > Y > G$ such that either
  $Y \subseteq \bar B(G,r)$ or there is $Z \subseteq Y$ such that
  $d(G,Z) > r-\varepsilon$ and
  $\RM_{r-\varepsilon}(Z) \geq \RM_r(G) = \RM_{r-\varepsilon}(G)$.

  In the second case we have
  $\RM_{r-\varepsilon}(G) = \RM_{r-\varepsilon}(Z) = \RM_{r-\varepsilon}(X)$, so
  from $d(G,Z) > r-\varepsilon$ we obtain:
  \begin{align*}
    \dM_{r-\varepsilon}(G)
    & =\dM_{r-\varepsilon}(X) \geq \dM_{r-\varepsilon}(G\cup Z) \\
    & \geq \dM_{r-\varepsilon}(G) + \dM_{r-\varepsilon}(Z) \\
    & > \dM_{r-\varepsilon}(G).
  \end{align*}
  This is impossible, so $\bar B(G,r) \supseteq Y > G$, as desired.
\end{proof}

\subsection{Further facts}

We prove here some additional facts regarding
Morley ranks and groups in an $\aleph_0$-stable theory $T$.

It was shown in \cite{BenYaacov-Usvyatsov:CFO} that for a stable
formula $\varphi$, the non forking
extensions of $\varphi$-types can be characterised as those having maximal
local Cantor-Bendixson ranks.
We shall show here that in
an $\aleph_0$-stable theory the same holds for global Cantor-Bendixson
ranks, i.e., for Morley ranks.

\begin{prp}
  \label{prp:NonDivMR}
  Assume $T$ is $\aleph_0$-stable, $A \subseteq B$, $q \in \tS_n(B)$,
  $p = q\rest_A \in \tS_n(A)$.
  Then the following are equivalent:
  \begin{enumerate}
  \item The type $q$ is a non forking extension of $p$.
  \item For all $\varepsilon > 0$:
    $\RM_\varepsilon(p) = \RM_\varepsilon(q)$.
  \item For all $\varepsilon > \varepsilon' > 0$:
    $\RM_\varepsilon(p) \leq \RM_{\varepsilon'}(q)$.
  \item For arbitrarily small $\varepsilon > 0$:
    $\RM_\varepsilon(p) = \RM_\varepsilon(q)$.
  \end{enumerate}
\end{prp}
\begin{proof}
  There is no harm in assuming that $B = \fM$.

  Let $\varepsilon > 0$.
  For an ordinal $\alpha$ recall that
  $\tS_n(\fM)^{(\alpha)}_\varepsilon$ denotes the $\alpha$th
  $\varepsilon$-Cantor-Bendixson derivative of $\tS_n(\fM)$ (i.e.,
  $(\varepsilon,f)$-Cantor-Bendixson derivative in the terminology of
  \cite{BenYaacov:TopometricSpacesAndPerturbations}.
  Let $\alpha_\varepsilon = \RM_\varepsilon(p)$, i.e., the maximal
  $\alpha$ such that
  $\tS_n(\fM)^{(\alpha)}_\varepsilon \cap [p] \neq \emptyset$.
  So let
  $X_\varepsilon
  = \tS_n(\fM)^{(\alpha_\varepsilon)}_\varepsilon \cap [p]$,
  the set of extensions
  of $p$ of maximal $\RM_\varepsilon$ rank.
  It is compact, and since $\alpha_\varepsilon$ is maximal it
  admits a cover by relatively open $\varepsilon$-finite
  subsets of $\tS_n(\fM)^{(\alpha_\varepsilon)}_\varepsilon$.
  This cover admits a finite sub-cover so $X_\varepsilon$ is $\varepsilon$-finite
  itself.
  \begin{cycprf}
  \item[\impnext]
    Observe that each $X_\varepsilon$ is invariant under the action of
    automorphisms of $\fM$ which fix $A$, and thus contains a non
    forking extension of $p$.
    As all non forking extensions of $p$
    are conjugate over $A$ they all belong to $X_\varepsilon$, i.e., they all
    satisfy $\RM_\varepsilon(q) = \alpha_\varepsilon = \RM_\varepsilon(p)$.
  \item[\impnext]
    As
    $\varepsilon > \varepsilon' \Longrightarrow \RM_\varepsilon(p)
    \leq \RM_{\varepsilon'}(p) = \RM_{\varepsilon'}(q)$.
  \item[\impnext]
    By \fref{lem:ConstRM} there are arbitrarily small pairs
    $\varepsilon > \varepsilon' > 0$ such that
    $\RM_\varepsilon(q) = \RM_{\varepsilon'}(q)$,
    whereby $\RM_\varepsilon(p) \leq \RM_\varepsilon(q)$.
    The inverse inequality is immediate.
  \item[\impfirst]
    Let
    $\cE =
    \{
    \varepsilon > 0\colon \RM_\varepsilon(p) = \RM_\varepsilon(q)
    \}$,
    so $\inf \cE = 0$.
    Let
    $X = \bigcap_{\varepsilon \in \cE} X_\varepsilon \subseteq [p]
    \subseteq \tS_n(\fM)$.
    On the
    one hand $X$ is compact and therefore metrically complete (see
    \cite{BenYaacov:TopometricSpacesAndPerturbations}).
    On the other hand it is $\varepsilon$-finite for
    every $\varepsilon$, i.e., totally bounded.
    Thus $X$ is metrically compact.
    It follows that for every $\varphi(\bar x,\bar y)$,
    the image of $X$ in
    $\tS_\varphi(\fM)$ is metrically compact
    (since $\varphi$ is uniformly continuous),
    so each $q \in X$ is definable over $\acl^{eq}(A)$.
    Thus every $q \in X$ is a non forking extension of $p$.
  \end{cycprf}
\end{proof}

Notice in passing that we showed that in a continuous $\aleph_0$-stable
theory every type has ``metrically compact multiplicity'', in analogy
with the finite multiplicity of types in a classical $\aleph_0$-stable theory.

\begin{lem}
  \label{lem:IsomMR}
  Let $f\colon X \to Y$ be a definable isometry between type-definable
  sets, and let $A$ be a set containing all the relevant parameters.
  Then
  $\RM_\varepsilon(a/A)
  \leq \RM_{\varepsilon'}(f(a)/A)
  \leq \RM_{\varepsilon''}(a/A)$
  for all $a \in X$ and all
  for all $\varepsilon > \varepsilon' > \varepsilon'' > 0$.
\end{lem}
\begin{proof}
  Proved using similar tools to those used to prove
  \fref{lem:ApproxCopyRM}.
\end{proof}

In the special case where $f$ is an isometric bijection between entire
sorts it induces an isometric homeomorphism of the corresponding
topometric type spaces and the ranks are equal by a direct argument.
It is not obvious that the ranks should be equal even in the spacial
case where $X$ is a definable set (and $Y$ its image, and therefore
definable as well).
The ranks would be equal if we had a positive answer for the
following:
\begin{qst}
  \label{qst:DefSetRelRM}
  Let $X$ be a definable set in $\cM$, say over $\emptyset$.
  Then the induced structure on $X$ is closed under quantification, so
  we may view  $X$ as a structure of its own right, and if $\cM$
  is $\aleph_0$-stable then so is $X$.

  In classical logic, types in the structure $X$
  have the same Morley ranks as the corresponding types in $\cM$.
  Is it true in continuous logic?
\end{qst}

\begin{thm}
  Assume $\langle G,S\rangle$ is a type-definable transitive group action
  with an invariant metric in an $\aleph_0$-stable continuous
  theory $T$, $p \in \tS_S(A)$.
  Then $p$ is generic if and
  only if
  $\RM_{\varepsilon'}(p) \geq \RM_\varepsilon(S)$
  for all $\varepsilon > \varepsilon' > 0$.
  If $S$ occupies an entire sort then
  this is further equivalent to
  $\RM_\varepsilon(p) = \RM_\varepsilon(S)$ for all $\varepsilon > 0$.
\end{thm}
\begin{proof}
  Along the same lines as the proof of
  \cite[Theorem~6.17]{BenYaacov:StableGroups},
  using \fref{prp:NonDivMR} and \fref{lem:IsomMR}.
  We leave the details to the reader.
\end{proof}

If the answer to \fref{qst:DefSetRelRM} is positive then the
assumption that $S$ occupies an entire sort is superfluous (since in
any case $S$ is definable).

\section{Definable metrics}
\label{sec:DefMet}

\subsection{Extensions of partial definable metrics}

Metrics and pseudo-metrics are a special (and distinguished) kind of
predicates.
\begin{dfn}
  \begin{enumerate}
  \item A \emph{definable (pseudo-)metric} is a definable predicate
    defining a (pseudo-)metric.
  \item Let $X$ be a type-definable set.
    A \emph{partial definable (pseudo-)metric} on $X$ is a
    definable partial predicate on $X \times X$ defining a (pseudo-)metric
    there.
  \end{enumerate}
  Again, when specifying the domain we drop the ``partial''.
  On the other hand, when wishing to make explicit that a definable
  metric is not partial, we say it is \emph{total}.
\end{dfn}

All type-definable metrics on a set $X$ are essentially the same.
This has first been proved in \cite{BenYaacov:Morley}, but since the
setting is somewhat different and the proof is short we repeat it.
\begin{lem}
  \label{lem:EquivDefMet}
  Let $d_1$ and $d_2$ be two definable metrics on a
  type-definable set $X$.
  Then they are uniformly equivalent.
\end{lem}
\begin{proof}
  Indeed, Let $\varepsilon > 0$, and consider the partial type
  \begin{gather*}
    \{x,y \in X\} \cup \{d_1(x,y) \geq \varepsilon\} \cup
    \{d_2(x,y) \leq 2^{-n}\colon n < \omega\}.
  \end{gather*}
  Since $d_2$ is a metric this partial type is inconsistent, and by
  compactness there exists $n < \omega$ such that
  $d_2(x,y) \leq 2^{-n} \Longrightarrow d_1(x,y) < \varepsilon$
  for all $x,y \in X$.
  Since this works for all $\varepsilon > 0$, and when exchanging the roles of
  $d_1$ and $d_2$, they are uniformly equivalent.
\end{proof}

\begin{lem}
  \label{lem:ChangeAmbientMetric}
  Let $\cM$ be a structure and $d_1$ a total definable metric
  on $\cM$.
  Let $\cM'$ be the structure obtained from $\cM$ by letting the
  standard metric on $\cM'$ be $d_1$ and relegating the original
  metric on $\cM$ to a new, non distinguished symbol $d_2$:
  \begin{gather*}
    d^{\cM'} = d_1^\cM, \qquad d_2^{\cM'} = d^\cM.
  \end{gather*}
  Then $M$ and $M'$ have the same definable sets and predicates.
\end{lem}
\begin{proof}
  Clearly, the two structures have the same definable predicates.
  The characterisation of definable sets
  via quantification in \fref{fct:DefSet} does not depend on
  the metric.
\end{proof}

While we know that every partial definable metric, say on $X$, extends
to a total definable predicate, it is not clear whether it can
extend to a total definable \emph{metric}.
We give some partial results in this direction.

The best result is under the assumption that $X$ is definable.
We start with a few lemmas that allow us to reduce the case of a
metric to the case of a pseudo-metric.
In fact the reduction step holds more generally for zero sets.

\begin{lem}
  \label{lem:PsMetJoin}
  Let $X$ be any set of points,
  $d_1$ and $d_2$ two pseudo-metrics on $X$.
  Then $d_1 \vee d_2$ is a pseudo-metric on $X$.
\end{lem}
\begin{proof}
  Let $d_3 = d_1 \vee d_2$.
  Clearly $d_3(x,x) = 0$ and $d_3(x,y) = d_3(y,x)$.
  For the triangle inequality, given $x,y,z \in X$ we may assume that
  $d_3(x,y) = d_1(x,y)$ in which case
  $d_3(x,y) \leq d_1(x,z) + d_1(y,z) \leq d_3(x,z) + d_3(y,z)$.
\end{proof}

\begin{lem}
  \label{lem:PsMetToMet}
  Let $X$ be a zero set.
  Then there exists a definable pseudo-metric $d_1$ with the following
  properties:
  \begin{enumerate}
  \item For all $x,y \in X$: $d_1(x,y) = 0$.
  \item If $x \neq y$ and $x \notin X$ then $d_1(x,y) > 0$.
  \end{enumerate}
  It follows that if $d_2$ is a definable
  pseudo-metric whose restriction to $X$ is a metric
  then $d_1 \vee d_2$ is a definable metric which agrees with $d_2$ on
  $X$.
\end{lem}
\begin{proof}
  Let $X$ be the zero set of $\varphi(x)$.
  Define:
  \begin{gather*}
    d_1(x,y)
    = \sup_z \left|
      \varphi(x)\wedge d(x,z) - \varphi(y)\wedge d(y,z)
    \right|.
  \end{gather*}
  Clearly, $d_1$ is a definable pseudo-metric and
  $x,y \in X \Longrightarrow d_1(x,y) = 0$.
  On the other hand, if $x \neq y$ and $x \notin X$ then
  $d_1(x,y) \geq \varphi(x)\wedge d(x,y) > 0$.

  Finally, assume that $d_2$ is a definable pseudo-metric which is a
  metric on $X$.
  By \fref{lem:PsMetJoin}, $d_1 \vee d_2$ is a (definable) pseudo-metric.
  The hypotheses now imply it is a metric and agrees with
  $d_2$ on $X$.
\end{proof}

\begin{prp}
  \label{prp:DefMetExt}
  Let $X$ be a definable set, $d_1$ a definable
  (pseudo-)metric on $X$.
  Then $d_1$ extends to a definable (pseudo-)metric.
\end{prp}
\begin{proof}
  Choose $\psi_1(x,y) \sqsupseteq d_1(x,y)$, and define:
  \begin{gather*}
    d_2(x,y) = \sup_{z\in X} |\psi_1(x,z) - \psi_1(y,z)|
  \end{gather*}
  Then $d_2$ is a definable pseudo-metric extending $d_1$.
  In case $d_1$ is a metric, $d_2$ is a definable pseudo-metric whose
  restriction to $X$ is a metric, so by \fref{lem:PsMetToMet} there is
  a definable metric $d_3$ extending $d_1$.
\end{proof}

We now move to considering the extension problem for partial definable
\hbox{(pseudo-)metrics} on type-definable sets.
We first observe this can be reduced to zero sets:

\begin{lem}
  \label{lem:ZeroSetMetExt}
  Assume $X$ is a type-definable set and $d_1$ a
  definable (pseudo-)metric on $X$.
  Then there exists a zero set $Y \supseteq X$
  such that $d_1$ extends to a
  definable (pseudo-)metric $d_2$ on $Y$.
\end{lem}
\begin{proof}
  In fact we prove something slightly stronger, namely that if
  $\psi_1 \sqsupseteq d_1$ is
  \emph{any} definable predicate extending $d_1$ then
  there exists a zero set $Y \supseteq X$ such that
  $\psi_1$ defines a (pseudo-)metric on $Y$.

  For this, let
  $\psi_2(x,y,z) =
  \psi_1(x,x) \vee (\psi_1(x,y) \dotminus \psi_1(x,z) \dotminus \psi_1(y,z))$.
  Then $\psi_2$ is zero on $X^3$, and by
  \fref{lem:ZeroSetByCoords} there is a zero set $Y \supseteq X$ such that
  $\psi_2$ is zero on $Y^3$.

  In case $d_1$ is a metric we have
  $d_1(x,y) = 0 \Longrightarrow d(x,y) = 0$ on $X^2$.
  By \fref{lem:ZeroSetImplication} (and \fref{lem:ZeroSetByCoords})
  there is a zero set $Y' \supseteq X$ such that
  $d_1(x,y) = 0 \Longrightarrow d(x,y) = 0$ for $x,y \in Y'$.
  Then $\psi_1$ defines a metric on the zero set $Y\cap Y' \supseteq X$.
\end{proof}

Joining \fref{lem:ZeroSetMetExt} with \fref{lem:PsMetToMet}, we reduce
the extension problems for partial metrics to the extension of
pseudo-metrics on zero sets.
Unfortunately, we do not know whether it is always true and suspect
that in its full generality it may fail.
The best approximation we have is:

\begin{lem}
  \label{lem:ApproxPsMetExt}
  Let $X$ be a zero set, $d_1$ a definable
  pseudo-metric on $X$.
  Then there exists a decreasing sequence of definable pseudo-metrics
  $d_{2,n}$ which converges uniformly to $d_1$ on $X$.
\end{lem}
\begin{proof}
  Let $X$ be the zero set of $\varphi$, and let
  $\psi_1(x,y) \sqsupseteq d_1(x,y)$.
  For $n < \omega$ let $\varphi_n(x) = 1 \dotminus 2^n\varphi(x)$: thus
  $\varphi(x) = 0 \Longrightarrow \varphi_n(x) = 1$, and
  $\varphi(x) \geq 2^{-n} \Longrightarrow \varphi_n(x) = 0$.
  Now let:
  \begin{gather*}
    d_{2,n}(x,y)
    = \sup_z \bigl|
      \varphi_n(z)\wedge\psi_1(x,z) - \varphi_n(z)\wedge\psi_1(y,z)
    \bigr|.
  \end{gather*}
  Clearly $d_{2,n}$ is a definable
  pseudo-metric for every $n$ and $d_{2,n} \geq d_{2,n+1}$.
  In addition, if $x,y \in X$ then choosing $z = x$ we see that
  $d_{2,n}(x,y) \geq d_1(x,y)$.

  Assume towards a contradiction that the pseudo-metrics
  $d_{2,n}$ do not converge uniformly to $d_1$ on $X$.
  Then there is some $\varepsilon > 0$ such that for all
  $n < \omega$ there are $a_n,b_n \in X$ and $c_n$ such that
  $|\varphi_n(c_n)\wedge\psi_1(a_n,c_n) - \varphi_n(c_n)\wedge\psi_1(b_n,c_n)|
  \geq d_1(a_n,b_n) + \varepsilon$.
  This implies in particular that
  $\varphi(c_n) \leq 2^{-n}$ and
  $|\psi_1(a_n,c_n) - \psi_1(b_n,c_n)| \geq d_1(a_n,b_n) + \varepsilon$.
  By compactness there are $a,b,c \in X$ such that 
  $|\psi_1(a,c) - \psi_1(b,c)| \geq d_1(a,b) + \varepsilon$, which is impossible as
  $\psi_1$ agrees with $d_1$ on $X$ and $d_1$ is a pseudo-metric.
\end{proof}

\begin{qst}
  Let $X$ be a zero set, $d_1$ a definable pseudo-metric on $X$.
  Can $d_1$ always be extended to a total definable pseudo-metric?
  Are there additional assumptions on the theory (e.g., stability or
  $\aleph_0$-stability) under which this is true?
\end{qst}

As it turns out, \fref{lem:ApproxPsMetExt} is too weak to be
useful for our purposes.
Having failed to solve the full metric extension problem,
we shall seek to solve weaker version thereof:
given a type-definable set $X$ equipped with a definable
metric $d_1$, we shall look for a total definable metric $d_2$ which
does not necessarily extend $d_1$, but which does preserve whatever
invariance properties $d_1$ may have.

For this purpose we shall use a few technical results from
\cite[Section~2.3]{BenYaacov:Morley}.
There we proved that in a Hausdorff cat in which there are not too
many ways for two elements to be distinct (technically: the
\emph{distance cofinality} is at most countable), there is a
definable metric.
For this purpose one first constructs a symmetric definable predicate
$\varphi(x,y)$ (a definable function in the terminology of
\cite{BenYaacov:Morley}, and denoted there by $h(x,y)$)
satisfying $\varphi(x,y) = 0 \Longleftrightarrow x=y$.
Then one needs to replace $\varphi$ with one which also satisfies the
triangle inequality.

In the context of an \emph{open} Hausdorff cat (which is essentially
the same
thing as a theory in continuous first order logic), one can just define
$d(x,y) = \sup_z |\varphi(x,z)-\varphi(y,z)|$, which is definable metric.
In the general case the metric
$\sup_z |\varphi(x,z)-\varphi(y,z)|$ need not be definable,
and a more complicated construction is required to extract
a definable metric from $\varphi$ directly without recourse to
quantification.
The same tools apply here.

\begin{fct}
  \label{fct:MetModConstr}
  Let $g\colon [0,1]^2 \to [0,1]$ be symmetric, non decreasing, and satisfy
  for all $u,w,t \in [0,1]$: $g(0,t) = t$ and if
  $g(u,w) < t$
  then there is $u < v \leq 1$ such that $g(v,w) < t$.

  Then there is a function $f\colon D \to [0,1]$,
  where $D = \{k2^{-n}\colon n < \omega, 0 \leq k \leq 2^n\}$
  denotes the set of all
  dyadic fractions in $[0,1]$, such that:
  \begin{enumerate}
  \item $f$ is strictly increasing.
  \item $f \leq \id_D$.
  \item For every $t,u \in D\cap(0,1]$: \qquad
    $t+u \leq 1 \Longrightarrow g(f(t),f(u)) < f(t+u)$.
  \end{enumerate}
\end{fct}
\begin{proof}
  This is \cite[Lemma~2.19]{BenYaacov:Morley}, with the sole difference
  being that we require $f(t) \leq t$ for all $t \in D$ rather than only
  for $t = \dyad{1}$.
  The only modification in the proof is in the construction of
  $f(\dyad{k})$ for odd $3 \leq k < 2^n$.
  There we have $f(\dyad{k-1}) \leq \dyad{k-1} < \dyad{k}$ by the
  induction hypothesis, and
  $f(\dyad{k-1}) < \min\{s',f(\dyad{k+1})\}$ as in the original
  proof, so we can choose $f(\dyad{k})$ such that
  $f(\dyad{k-1}) < f(\dyad{k}) < \min\{s',f(\dyad{k+1}),\dyad{k}\}$.
\end{proof}

The following result was implicit in \cite[Section~2.3]{BenYaacov:Morley}
(modulo the slight improvement above to the requirement on $f$):
\begin{prp}
  \label{prp:MetMod}
  Let $\varphi(x,y)$ be a symmetric and reflexive definable predicate,
  by which we mean that $\varphi(x,y) = \varphi(y,x)$ and $\varphi(x,x) = 0$.
  Then there is a continuous increasing function $h\colon[0,1] \to [0,1]$,
  satisfying $h(0) = 0$ and $h(x) \geq x$, such that
  $h \circ \varphi(x,y)$ is a definable pseudo-metric.
  If in addition $x\neq y \Longrightarrow \varphi(x,y) > 0$ then
  $h \circ \varphi(x,y)$ is a definable metric.
\end{prp}
\begin{proof}
  Define $g\colon [0,1]^2 \to [0,1]$ by:
  \begin{gather*}
    g(t,u) = \sup\{\varphi(x,y)\colon \exists z\, \varphi(x,z)\leq t\wedge \varphi(y,z)\leq u\}.
  \end{gather*}
  As in the Claim following the proof of
  \cite[Lemma~2.19]{BenYaacov:Morley}, $g$ satisfies the assumptions of
  \fref{fct:MetModConstr}.
  Thus there is a function $f\colon D \to [0,1]$ such that:
  \begin{enumerate}
  \item $f$ is strictly increasing.
  \item $f \leq \id_D$.
  \item For every $t,u \in D\cap(0,1]$:
    \begin{gather*}
      t+u \leq 1 \Longrightarrow g(f(t),f(u)) < f(t+u)
    \end{gather*}
  \end{enumerate}
  We define $h\colon [0,1] \to [0,1]$ to be a weak inverse of $f$:
  $h(t) = \sup\{u \in D\colon f(u) < t\}$.
  As $f$ is strictly increasing it follows that
  $h(t) = \inf\{u \in D\colon f(u) > t\}$, and that $h$ is continuous and
  weakly increasing.
  In addition, $f \leq \id_D \Longrightarrow h \geq \id_{[0,1]}$.

  Let $\psi(x,y) = h \circ \varphi(x,y)$, which is a definable predicate by
  continuity of $h$.
  Clearly $\psi$ is reflexive and symmetric as well, and it is left to
  show that it satisfies the triangle inequality.
  Indeed, assume not, that is
  $\psi(x,y) > \psi(x,z) + \psi(y,z) + \varepsilon$ for some
  $x,y,z$.
  As $\psi(x,z) + \varepsilon/2 > \psi(x,z) = h\circ\varphi(x,z)$,
  there is $u \in D$, $u < \psi(x,z) + \varepsilon/2$, such that
  $f(u) > \varphi(x,z)$.
  Similarly there is $w < \psi(y,z) + \varepsilon/2$ in $D$ such that
  $f(w) > \varphi(y,z)$.
  As $u+w < \psi(x,y) \leq 1$ we have:
  \begin{gather*}
    \varphi(x,y)
    \leq g(\varphi(x,z),\varphi(y,z))
    \leq g(f(u),f(w))
    < f(u+w) < f(\psi(x,y)).
  \end{gather*}
  In other words we have $\varphi(x,y) < (f\circ h)(\varphi(x,y))$,
  contradicting
  the definition of $h$.

  We have shown that $h\circ\varphi(x,y)$ is a pseudo-metric.
  If $\varphi$ satisfies $\varphi(x,y) = 0 \Longrightarrow x=y$, so does
  $h\circ\varphi$ (as $h \geq \id$), which is therefore a metric.
\end{proof}

Putting it all together we get:
\begin{thm}
  \label{thm:MetModExt}
  Let $X$ be a type-definable set, $d_1$ a partial definable
  (pseudo-)metric on $X$.
  Then there exist a total definable
  (pseudo-)metric $d_2$ and a continuous
  increasing function $h\colon [0,1] \to [0,1]$ such that
  $h(0) = 0$, $h \geq \id$ and $d_2 \sqsupseteq h \circ d_1$.
\end{thm}
\begin{proof}
  First, by \fref{lem:ZeroSetMetExt} we may assume $X$ is a zero set,
  defined by $\varphi(x) = 0$.
  Let $\psi_1 \sqsupseteq d_1$, and define:
  \begin{align*}
    \psi_2(x,y) & = (\psi_1(x,y)\wedge\psi_1(y,x))
    \dotminus \psi_1(x,x) \dotminus \psi_1(y,y), \\
    \psi_3(x,y) & = d(x,y) \wedge (\varphi(x)\vee\varphi(y)), \\
    \psi_4(x,y) & = \psi_2 \vee \psi_3.
  \end{align*}
  Then $\psi_2 \sqsupseteq d_1$ as well and $\psi_3\rest_{X^2}$ is zero, whereby
  $\psi_4 \sqsupseteq d_1$.
  In addition $\psi_2$, $\psi_3$ and $\psi_4$ are all symmetric and reflexive,
  and we may apply \fref{prp:MetMod} to find $h$ as such that
  $d_2 = h \circ \psi_4$ is a total pseudo-metric.

  Assume now that $d_1$ is a metric.
  Then for $x\neq y$,
  if $x,y \in X$ then $\psi_2(x,y) = d_1(x,y) \neq 0$,
  and otherwise $\psi_3(x,y) \neq 0$, so either way $\psi_4(x,y) \neq 0$.
  Thus $d_2$ is a total metric.
\end{proof}

\subsection{Existence of invariant metrics}

In general type-definable groups need not be invariant.
For definable groups this is easily resolved:
\begin{prp}
  \label{prp:DefGrpInvMet}
  Let $\langle G,\cdot\rangle$ be a definable group.
  Then $G$ admits a total definable invariant metric
  (i.e., a total definable metric on the sort of $G$ which is
  invariant on $G$).
\end{prp}
\begin{proof}
  For simplicity assume $G$ is definable in the home sort.
  Since $G$ is a definable set, the predicate $d(x,G)$ is definable,
  and we may quantify over members of $G$.
  Let:
  \begin{align*}
    \psi_0(u,v,x,y) & \sqsupseteq d(uxv,uyv) \\
    \psi_1(x,y) & = \sup_{u,v \in G} \psi_0(u,v,x,y) \\
    d_1(x,y) & = \psi_1(x,y)\rest_{x,y \in G}
  \end{align*}
  This means that for $x,y \in G$ we have:
  $d_1(x,y) = \sup_{u,v \in G} d(uxv,uyv)$.
  This is easily verified to define an invariant metric on
  $G$, which can be extended to a total definable metric by
  \fref{prp:DefMetExt}.
\end{proof}

In the case of a type-definable group we only have partial results.
First, as before we can reduce any problem to the case of zero sets:

\begin{lem}
  \label{lem:GrpZeroSet}
  Let $G$ be a type-definable group.
  Then there exists a type-definable supergroup $G' \geq G$ whose domain
  is a zero set.
  Moreover, if $d_1$ is an invariant definable metric on $G$, and
  $\psi_1 \sqsupseteq d_1$,
  we can choose $G'$ so that $d_2 = \psi_1\rest_{G'}$ is an
  invariant metric on $G'$.
\end{lem}
\begin{proof}
  Using \fref{prp:PartFuncClosedZeroSetExt}
  we find a zero set $X_0 \supseteq G$ to which the law of $G$ and the
  inverse mappings extend to definable mappings $\cdot'\colon X_0^2 \to X_0$ and
  $^{-1\prime}\colon X_0 \to X_0$.
  The satisfaction of the identity $(x\cdot'y)\cdot'z = x\cdot'(y\cdot'z)$
  (i.e., $d((x\cdot'y)\cdot'z,x\cdot'(y\cdot'z)) = 0$) in $X_0$
  translates a condition of the form $\varphi_{ass}(x,y,z) = 0$.
  By \fref{lem:ZeroSetByCoords} there
  is a zero set $X_1$ such that $X_0 \supseteq X_1 \supseteq G$ and such that
  $\varphi_{ass}$ is zero on $X_1^3$.
  We take care of $x\cdot'e = x$ and $x\cdot'x^{-1\prime} = e$ similarly.
  We end up with a zero set $X_0 \supseteq X_2 \supseteq G$ on which all the
  identities above hold.
  Now $X_2$ needs not be closed under product and inverse, but that is
  taken care of by a second application of
  \fref{prp:PartFuncClosedZeroSetExt}.
\end{proof}

\begin{rmk}
  This argument would work for any kind of functional structure
  defined by a set of identities (e.g., rings).
  Adding \fref{lem:ZeroSetImplication} we can further extend the
  argument to structures whose definition involves an implication,
  such as integral domains.
  We have already seen an example of that in
  \fref{lem:ZeroSetMetExt}.
\end{rmk}

Alternatively, given a partial invariant definable metric on $G$ we
can extend it to a total definable metric up to some uniform
modification which preserves the invariance.

\begin{prp}
  \label{prp:GrpInvMetModExt}
  Let $G$ be a type-definable group admitting a partial invariant
  definable metric.
  Then it also admits a total one.
\end{prp}
\begin{proof}
  Just apply \fref{thm:MetModExt} to the partial invariant metric.
\end{proof}

In a stable theory we obtain a partial invariant metric via generic
translations.
We first recall a few facts regarding generic elements and types in
stable groups:

\begin{dfn}
  Let $G$ be a type-definable group in a stable theory,
  say over a parameter set $A$.
  \begin{enumerate}
  \item 
    We say that $G$ is \emph{connected} if it has no proper
    type-definable subgroups of bounded index (over any parameters).
  \item Let $B \supseteq A$.
    An element $g \in G$ is called \emph{generic}
    over $B$ if for every $h \in G$:
    \begin{gather*}
      g \ind_B h \quad \Longrightarrow \quad gh \ind_A B,h.
    \end{gather*}
    (This notion is sometimes called \emph{dividing-generic} and is also
    useful sense in simple theories.
    In stable theories it coincides with other notions of genericity.)
  \item Let $\tS_G(B) = [x \in G] \subseteq \tS_x(B)$,
    namely the set of complete types over $B$ of members of
    $G$.
    A type $p \in \tS_G(B)$ is generic if its realisations are.
  \end{enumerate}
\end{dfn}

\begin{fct}
  \label{fct:Geenric}
  Let $G$ be a type-definable group over a set $A$
  in a stable theory,
  and let $B \supseteq A$.
  \begin{enumerate}
  \item Generic elements over $B$ exist.
  \item An element $g$ is generic over $B$ if and only if
    $g^{-1}$ is.
    It follows that if $g$ is generic over $B$ and
    $g \ind_B h$ then $hg \ind_A h,B$ as well.
  \item An element $g \in G$ is generic over $B$
    if and only if it is generic over $A$ and $g \ind_A B$.
  \item If $g,h \in G$, $g$ is generic over $B$ and
    $g \ind_B h$ then $gh$ and $hg$ are both generic over $B$.
  \end{enumerate}
\end{fct}
\begin{proof}
  See \cite[Theorem~6.10]{BenYaacov:StableGroups}.
  (In fact, this holds in the more general setting of a type-definable
  group in a thick simple cat, see
  \cite[Section~1.3]{BenYaacov:ThicknessAndCatTSF}.)
\end{proof}

\begin{fct}
  \label{fct:ConnComp}
  Let $G$ be a type-definable group with parameters in some set $A$
  in a stable theory.
  Then $G$ admits a smallest type-definable subgroup of bounded index
  $G^0$, called the \emph{connected component} of $G$.
  It has the following properties:
  \begin{enumerate}
  \item The connected component $G^0$ is a connected normal subgroup of
    $G$, and is type-definable over $A$ as well.
  \item If $G$ is connected the it admits a unique generic type
    over $A$, which is in addition stationary.
    It follows that $G$ admits a unique generic type over every set
    $B \supseteq A$, namely $p\rest^B$, the unique non forking
    extension of $p$ to $B$.
  \end{enumerate}
\end{fct}
\begin{proof}
  See \cite[Theorem~6.14]{BenYaacov:StableGroups}.
\end{proof}

\begin{thm}
  \label{thm:StabConnGrpInvMet}
  Assume $G$ is a connected type-definable group in a stable theory.
  Then there exists a total metric which is invariant on $G$.
\end{thm}
\begin{proof}
  Let $p(x)$ be the unique generic type of $G$.
  stationary as well.
  Again let $\psi_0(u,v;x,y) \sqsupseteq d(uxv,uyv)$.
  As $(p \otimes p)(u,v)$
  (the free product of $p$ with itself)
  is a definable stationary type over the same parameters
  as $G$, let $\psi_1(x,y)$ be its $\psi_0$-definition.
  Thus, if $x,y \in G$ and $u,v$ are any independent generic elements
  over $x,y$, then $\psi_1(x,y) = d(uxv,uyv)$.

  Let $d_1(x,y) = \psi_1(x,y)\rest_{G^2}$, and we claim it is an
  invariant metric on $G$.
  To verify it is a metric, let $x,y,z \in G$ be any three elements.
  Then choosing $u,v$ to be independent generics over $x,y,z$ we can
  verify the metric axioms for this triplet.
  We also get that $uz,v$ are two independent generics over $x,y$,
  whereby $d_1(zx,zy) = d(uzxv,uzyv) = d_1(x,y)$.
  Right-invariance is verified similarly.

  Apply \fref{prp:GrpInvMetModExt} to conclude.
\end{proof}

Notice that if we knew how to show, as in classical first order logic,
that in a stable continuous theory every type-definable group is
contained in a definable one, we would obtain an alternative proof for
\fref{thm:StabConnGrpInvMet} using \fref{prp:DefGrpInvMet},
without the connectedness assumption.

\section{The main theorem (full version)}
\label{sec:MainThmFull}

\begin{thm}
  \label{thm:MainThm}
  A type-definable group in an $\aleph_0$-stable theory
  is definable.
\end{thm}
\begin{proof}
  By \fref{prp:BddIdxDefGrp} it will suffice to show that the
  connected component of $G$ is definable, so
  we may assume that $G$ is connected.
  Then by \fref{thm:StabConnGrpInvMet} there exists a total definable
  metric $d_1$ which is invariant on $G$.
  By \fref{lem:ChangeAmbientMetric} we assume $d_1$ is the ambient
  metric.
  Now apply \fref{thm:MainThmApprox}.
\end{proof}

Notice that in the proof we needed to pass to the connected component
since we do not know yet (although it seems plausible) whether a general
type-definable group in a stable theory admits an invariant metric.
Moreover, the passage from an partial invariant definable metric on the
connected component to a total one required allowing a modification
to that metric.
\textit{A posteriori} we have:

\begin{cor}
  \label{cor:MainThm}
  Let $G$ be a type-definable group in an $\aleph_0$-stable theory.
  Then:
  \begin{enumerate}
  \item $G$ admits an invariant metric.
  \item Every partial metric on $G$ extends to a total one.
  \end{enumerate}
\end{cor}
\begin{proof}
  By the main theorem, $G$ is definable, so just apply
   \fref{prp:DefGrpInvMet} and \fref{prp:DefMetExt}.
\end{proof}

\section{An application: descending chain conditions}
\label{sec:ChainConditions}

\begin{dfn}
  \label{dfn:ApproxStab}
  Let $\{X_\alpha\}_{\alpha < \lambda}$
  be a descending sequence of closed sets in a metric space.
  We say that the sequence \emph{approximately stabilises}
  if for every $\varepsilon > 0$ there exists
  $\alpha_\varepsilon < \lambda$ such that
  $X_{\alpha_\varepsilon} \subseteq B(X_\beta,\varepsilon)$
  for all $\beta < \lambda$.
\end{dfn}
Notice an approximately stabilising sequence of closed sets
whose length has uncountable cofinality necessarily stabilises.

\begin{lem}
  \label{lem:StrApproxStab}
  Assume $\{X_\alpha\}_{\alpha < \lambda}$ is a descending
  sequence of closed subsets of a complete metric space
  (e.g., type-definable subsets of a continuous structure),
  and let $X = \bigcap_\alpha X_\alpha$.
  Then $\{X_\alpha\}_{\alpha < \lambda}$
  approximately stabilises if and only if
  for all $\varepsilon > 0$ there is $\alpha_\varepsilon$ such that
  $X_{\alpha_\varepsilon} \subseteq B(X,\varepsilon)$.
\end{lem}
\begin{proof}
  Right to left is immediate, so we prove left to right.
  In addition, if $\lambda$ has uncountable cofinality then
  the sequence stabilises precisely, i.e.,
  $X = X_\alpha$ for some (all) big enough $\alpha$.
  We are left with the case of a sequence
  $\{X_n\}_{n < \omega}$.

  By assumption, for each $m$ there is $n_m$ such that
  $X_{n_m} \subseteq B(X_n,2^{-m})$ for all $n$, and we may assume
  that $m \leq n_m$ for all $m$.
  Let $m_0$ be such that $2^{-m_0} < \varepsilon$.
  We claim that
  $X_{n_{m_0+1}} \subseteq B(X,\varepsilon)$.
  Indeed, let $x_0 \in X_{n_{m_0+1}}$.
  Given $x_k \in X_{n_{m_0+k+1}}$, let $x_{k+1} \in X_{n_{m_0+k+2}}$ be
  such that $d(x_k,x_{k+1}) < 2^{-m_0-k-1}$.
  Then the sequence $\{x_k\}$ converges to some point $x$, satisfying
  $d(x,x_k) < 2^{-m_0-k}$.
  Notice that for all $n$ we have $x_k \in X_n$ for $k$ large enough.
  Thus $x \in X_n$ for all $n$, whereby $x \in X$.
  We conclude that $d(x_0,X) < 2^{-m_0} < \varepsilon$, whence
  $X_{n_{m_0+1}} \subseteq B(X,\varepsilon)$, as desired.
\end{proof}

\begin{lem}
  \label{lem:ApproxStabDef}
  Let $\{X_\alpha\}_{\alpha<\lambda}$ be a descending sequence of definable
  sets.
  If the chain approximately stabilises (in any model) then
  $X = \bigcap X_\alpha$ is a definable set.
  The converse holds in sufficiently saturated models.
\end{lem}
\begin{proof}
  Assume first the sequence approximately stabilises.
  By \fref{lem:StrApproxStab} we have
  $X_{\alpha_\varepsilon} \subseteq B(X,\varepsilon)$
  for all $\varepsilon> 0$, whereby
  $X \subseteq X_{\alpha_{\varepsilon/2}}
  \subseteq B(X_{\alpha_{\varepsilon/2}},\varepsilon/2)
  \subseteq B(X,\varepsilon)$
  for all $\varepsilon > 0$.
  Now apply the criterion for definability in
  \fref{fct:DefSet}.\fref{item:DefSetMetricNeighbourhood}
  twice.
  First, since $X_{\varepsilon/2}$ is definable there is a formula
  $\psi_\varepsilon$ such that
  \begin{gather*}
    X \subseteq X_{\alpha_{\varepsilon/2}}
    \subseteq
    \bigl\{ \bar a \in M^n\colon \psi_\varepsilon(\bar a) = 0 \bigr\}
    \subseteq
    \bigl\{ \bar a \in M^n\colon \psi_\varepsilon(\bar a) < 1 \bigr\}
    \subseteq
    B(X_{\alpha_{\varepsilon/2}},\varepsilon/2)
    \subseteq B(X,\varepsilon).
  \end{gather*}
  A second application shows that $X$ is definable.

  Conversely, assume $X$ is definable, and that the equality
  $X = \bigcap X_\alpha$ holds in a sufficiently saturated model.
  Then by compactness, for all $\varepsilon > 0$ there is
  $\alpha_\varepsilon$ such that
  $\{x \in X_{\alpha_\varepsilon}\} \cup \{d(x,X) \geq \varepsilon\}$
  is contradictory, whereby
  $X_{\alpha_\varepsilon} \subseteq B(X,\varepsilon)$.
\end{proof}

\begin{dfn}
  We say that a class $\cC$ of sets satisfies the
  \emph{metric descending chain condition (MDCC)} if every descending
  chain in $\cC$ approximately stabilises.
\end{dfn}

\begin{thm}
  \label{thm:OStabMDCC}
  Let $G$ be a type-definable group in an $\aleph_0$-stable structure,
  $\cG$ the class of type-definable subgroups of $G$.
  Then $\cG$ satisfies the MDCC.
\end{thm}
\begin{proof}
  Let $\{G_\alpha\}_{\alpha<\lambda} \subseteq \cG$
  be a descending chain of type-definable
  groups, and let $G_\infty = \bigcap G_\alpha$.
  Then $G_\infty$ is definable by the main theorem.
  If necessary pass to a sufficiently saturated elementary extension
  of the ambient model and apply the converse part of
  \fref{lem:ApproxStabDef}.
\end{proof}

We conclude with a relatively easy chain condition for arbitrary
stable theories (and which has nothing to do with our main theorem).

\begin{thm}
  \label{thm:UnifDefGrpMDCC}
  Let $\cX$ be a family of uniformly definable sets in a stable
  theory, meaning that there is a definable predicate
  $\varphi(x,y)$ such that for every $X \in \cX$ the predicate
  $d(x,X)$ is an instance $\varphi(x,a_X)$.
  Then $\cX$ satisfies the MDCC.

  Moreover, if $X$ is the intersection of any descending chain in
  $\cX$ then $d(x,X)$ is also definable by an instance of
  $\varphi$.
\end{thm}
\begin{proof}
  By \fref{lem:ApproxStabDef} it is enough to prove the moreover part.
  Assume that $\{X_\alpha\}_{\alpha<\lambda} \subseteq \cX$
  is a descending chain, and let
  $d(x,X_\alpha) = \varphi(x,a_\alpha)$.
  The sequence of definable predicates $\varphi(x,a_\alpha)$ is
  increasing, and by definition of stability (no order property) it
  must converge uniformly.
  It follows by an easy compactness argument that
  there exists a parameter $a$ such that
  $\varphi(x,a) = \lim_\alpha \varphi(x,a_\alpha)$ uniformly,
  and its zero set is necessarily
  $X = \bigcap_\alpha X_\alpha$.
  Moreover, by \fref{prp:DefSetParam},
  we may further arrange that $\models \Sigma_\psi(a)$,
  i.e., that $\varphi(x,a) = d(x,X)$.
\end{proof}

\providecommand{\bysame}{\leavevmode\hbox to3em{\hrulefill}\thinspace}
\providecommand{\MR}{\relax\ifhmode\unskip\space\fi MR }
\providecommand{\MRhref}[2]{%
  \href{http://www.ams.org/mathscinet-getitem?mr=#1}{#2}
}
\providecommand{\href}[2]{#2}


\begin{thebibliography}{BBHU08}

\bibitem[BBHU08]{BenYaacov-Berenstein-Henson-Usvyatsov:NewtonMS}
Itaï {Ben Yaacov}, Alexander Berenstein, C.~Ward Henson, and Alexander
  Usvyatsov, \href{http://math.univ-lyon1.fr/~begnac/articles/mtfms.pdf}
  {\emph{Model theory for metric structures}}, Model theory with Applications
  to Algebra and Analysis, volume 2 (Zoé Chatzidakis, Dugald Macpherson, Anand
  Pillay, and Alex Wilkie, eds.), London Math Society Lecture Note Series, vol.
  350, Cambridge University Press, 2008, pp.~315--427.

\bibitem[{Ben}]{BenYaacov:StableGroups}
Itaï {Ben Yaacov},
  \href{http://math.univ-lyon1.fr/~begnac/articles/StabGrps.pdf}
  {\emph{Stability and stable groups in continuous logic}}, submitted,
  \href{http://arxiv.org/abs/0810.4087}{arXiv:0810.4087}.

\bibitem[{Ben}03]{BenYaacov:ThicknessAndCatTSF}
\bysame, \href{http://dx.doi.org/10.4064/fm179-3-2} {\emph{Thickness, and a
  categoric view of type-space functors}}, Fundamenta Mathematic{\ae}
  \textbf{179} (2003), 199--224.

\bibitem[{Ben}05]{BenYaacov:Morley}
\bysame, \href{http://dx.doi.org/10.2178/jsl/1122038916} {\emph{Uncountable
  dense categoricity in cats}}, Journal of Symbolic Logic \textbf{70} (2005),
  no.~3, 829--860.

\bibitem[{Ben}08]{BenYaacov:TopometricSpacesAndPerturbations}
\bysame, \href{http://dx.doi.org/10.1007/s11813-008-0009-x} {\emph{Topometric
  spaces and perturbations of metric structures}}, Logic and Analysis
  \textbf{1} (2008), no.~3--4, 235--272,
  \href{http://arxiv.org/abs/0802.4458}{arXiv:0802.4458}.

\bibitem[Ber06]{Berenstein:DefinableSubgroupsOfMeasureAlgebras}
Alexander Berenstein, \emph{Definable subgroups of measure algebras},
  Mathematical Logic Quarterly \textbf{52} (2006), no.~4, 367--374.

\bibitem[BU]{BenYaacov-Usvyatsov:CFO}
Itaï {Ben Yaacov} and Alexander Usvyatsov,
  \href{http://math.univ-lyon1.fr/~begnac/articles/cfo.pdf} {\emph{Continuous
  first order logic and local stability}}, Transactions of the American
  Mathematical Society, to appear,
  \href{http://arxiv.org/abs/0801.4303}{arXiv:0801.4303}.

\end{thebibliography}
\end{document}